\documentclass[12pt]{amsart}
\usepackage{amscd}
\usepackage{amsmath}
\usepackage{verbatim}
\usepackage[normalem]{ulem}
\usepackage{longtable}

\usepackage[all, cmtip,arrow]{xy}

\usepackage{pb-diagram,pb-xy}
\usepackage{marginnote}

\usepackage{amssymb}

\numberwithin{equation}{section}

\newtheorem{thm}[equation]{Theorem}
\newtheorem{prop}[equation]{Proposition}

\newtheorem{lem}[equation]{Lemma}
\newtheorem{cor}[equation]{Corollary}

\theoremstyle{definition}
\newtheorem{rem}[equation]{Remark}
\newtheorem{example}[equation]{Example}
\newtheorem{dfn}[equation]{Definition}

\newtheorem{ntt}[equation]{}

\newcommand{\codim}{\operatorname{codim}}

\newcommand{\U}{\mathcal{U}}

\newcommand{\SB}{\mathop{\mathrm{SB}}}

\newcommand{\Br}{\mathop{\mathrm{Br}}}

\renewcommand{\Im}{\mathop{\mathrm{Im}}}
\newcommand{\Pic}{\mathop{\mathrm{Pic}}}
\newcommand{\ind}{\mathop{\mathrm{ind}}}

\newcommand{\CH}{\mathop{\mathrm{CH}}\nolimits}

\newcommand{\PGSp}{\operatorname{\mathrm{PGSp}}}

\newcommand{\Spin}{\operatorname{\mathrm{Spin}}}
\newcommand{\PGU}{\operatorname{\mathrm{PGU}}}

\newcommand{\PGO}{\operatorname{\mathrm{PGO}}}

\newcommand{\PGL}{\operatorname{\mathrm{PGL}}}

\newcommand{\SL}{\operatorname{\mathrm{SL}}}

\newcommand{\E}{\mathrm{E}}
\newcommand{\ff}{\mathbb{F}}

\newcommand{\OGr}{\mathrm{OGr}}

\newcommand{\Ch}{\mathop{\mathrm{Ch}}\nolimits}

\newcommand{\res}{\mathop{\mathrm{res}}\nolimits}
\newcommand{\cores}{\mathop{\mathrm{cor}}\nolimits}

\newcommand{\pr}{\operatorname{\mathrm{pr}}}

\newcommand{\mult}{\operatorname{mult}}

\newcommand{\Char}{\mathop{\mathrm{char}}\nolimits}

\newcommand{\zz}{\mathbb{Z}}
\newcommand{\Z}{\mathbb{Z}}
\newcommand{\F}{\mathrm{F}}
\newcommand{\G}{\mathrm{G}}
\newcommand{\A}{\mathrm{A}}
\newcommand{\B}{\mathrm{B}}
\newcommand{\PP}{\mathbb{P}}
\newcommand{\qq}{\mathbb{Q}}

\newcommand{\C}{\mathrm{C}}

\newcommand{\Spec}{\operatorname{Spec}}
\newcommand{\End}{\operatorname{End}}

\newcommand{\Corr}{\operatorname{Corr}}
\newcommand{\Aut}{\operatorname{Aut}}

\newcommand{\pt}{\mathrm{pt}}

\newcommand{\Sum}{\operatornamewithlimits{\textstyle\sum}}
\newcommand{\Oplus}{\operatornamewithlimits{\textstyle\bigoplus}}

\newcommand{\disc}{\operatorname{disc}}

\newcommand{\CM}{\operatorname{CM}}

\newcommand{\D}{\mathrm{D}}

\newcommand{\BC}{\mathrm{BC}}

\renewcommand{\phi}{\varphi}

\title
[$J$-invariant of linear algebraic groups of outer type]
{$J$-invariant of linear algebraic groups of outer type}


\author
[Nikita Geldhauser]
{Nikita Geldhauser}

\author
[Maksim Zhykhovich]
{Maksim Zhykhovich}

\address{Geldhauser:
Mathematisches Institut der Universit\"at M\"unchen, Theresienstr. 39, D-80333 M\"unchen, Germany}
\email{geldhauser@math.lmu.de}

\address{Zhykhovich:
Mathematisches Institut der Universit\"at M\"unchen, Theresienstr. 39, D-80333 M\"unchen, Germany}
\email{zhykhovich@math.lmu.de}

\thanks{The work of the authors is supported by the DFG research grant SE 1721/4-1.}

\begin{document}

\begin{abstract}
We extend the notion of the $J$-invariant to arbitrary semisimple linear algebraic groups and provide complete decompositions for the normed Chow motives of all generically quasi-split twisted flag varieties. Besides, we establish some combinatorial patterns for normed Chow groups and motives and provide some explicit formulae for values of the $J$-invariant. 
\end{abstract}

\maketitle

\section{Introduction}

In the present article we investigate the Chow motives of twisted flag varieties for semisimple linear algebraic groups of outer type.

Chow motives were introduced by Alexander Grothendieck in the 1960s and since then they became a fundamental tool for investigating the structure of algebraic varieties.
Motives have applications to classical mathematical problems. For example, Voevodsky's proof of the Milnor conjecture \cite{Vo03b} relies on Rost's computation of
the motive of a Pfister quadric \cite{Ro98}.

More recently Geldhauser, Petrov, and Zainoulline established the structure of Chow motives of generically
split twisted flag varieties and introduced an invariant of algebraic groups, called the $J$-invariant (see \cite{PSZ08}, \cite{PS10}, \cite{PS12}).
In the case of quadratic forms an equivalent
invariant was introduced previously by Vishik in \cite{Vi05}.

The $J$-invariant was an important tool to solve several long-standing problems. For example, it plays an important role in the progress on the Kaplansky problem about possible values of the $u$-invariant of fields by Vishik \cite{Vi07} and in the solution of a problem of Serre about groups of type $\E_8$ and its finite subgroups \cite{S16}, \cite{GS10}. More recently, Geldhauser and Petrov generalized the $J$-invariant for groups of inner type to arbitrary oriented cohomology theories in the sense of Levine--Morel \cite{LM} satisfying some axioms (see \cite{PS21}). 

The $J$-invariant is a discrete invariant which describes the
motivic behavior of the variety of Borel subgroups of a semisimple linear algebraic group. 
Let $G$ be a split semisimple group over a field $F$, let $B$ be a Borel subgroup of $G$, let $E$ be a $G$-torsor, and let $p$ be a prime number. 
It turns out that the Chow motive of the twisted flag variety $E/B$ with coefficients in $\ff_p$ decomposes into a direct sum of Tate twists of an indecomposable motive $R_p(E)$ and the Poincar\'e polynomial of $R_p(E)$ over a splitting field of $E$ is a product of ``cyclotomic-like'' polynomials. More precisely, it equals 

\begin{equation}\label{f11}
\prod_{i=1}^r\frac{t^{d_ip^{j_i}}-1}{t^{d_i}-1}\in\zz[t],
\end{equation}
for some integers $r$, $d_i$ and $j_i$. The parameters $r$ and $d_i$ are combinatorial and are contained essentially in a table of Kac \cite{Kac85} and the $r$-tuple
$(j_1,\ldots,j_r)$ is of geometric nature and constitutes the $J$-invariant of $E$.

For groups of outer type, i.e., when $G$ is quasi-split, but not split, the situation is more subtle. Recently Fino introduced the $J$-invariant for Hermitian forms (see \cite{Fi19}).  Besides, Victor Petrov announced several years ago a computation of Chow motives of the varieties of Borel subgroups for exceptional groups of outer type.

In the present article we extend the $J$-invariant to all semisimple algebraic groups. Our idea was that it should be an analogous direct summand $R_p(E)$ for groups of outer type whose Poincar\'e polynomial has cyclotomic type~\eqref{f11}. Note first of all that it is not true anymore in general that if $G$ is a group of outer type, then the Chow motive of $E/B$ is a direct sum of Tate twists of the same indecomposable motive $R_p(G)$. Moreover, our motivic computations showed that the most numerological patterns, which occur for groups of inner type, do not apply for Chow motives for groups of outer type.

It turns out that for groups of outer type it is more natural to consider the ``normed'' Chow motives instead. Such motives have been already considered before in the literature, notably by Karpenko and Fino. Moreover, there are parallels between normed Chow motives and isotropic Chow motives of Vishik \cite{Vi24}. We recall the definition of normed Chow rings and motives in Section~\ref{sec5}.

The classical Solomon theorem for groups of inner type allows to compute the Poincar\'e polynomials of split flag varieties (see Proposition~\ref{solo}a). For groups of outer type there are analogous formulae for normed Chow groups (see Proposition~\ref{solo}b,c). Actually, the formulae for normed Chow groups are essentially contained in the classical book of Carter \cite{Car} and 
coincide in essence with formulae for orders of twisted finite simple groups of Lie type.

In Section~\ref{sec7} we introduce the $J$-invariant for groups of outer type and provide a decomposition for the normed Chow motive of $E/B$ (see Theorem~\ref{prop612}). Note that the normed Chow motive of $E/B$ decomposes into a direct sum of Tate twists of the same indecomposable motive and  its Poincar\'e polynomial indeed has cyclotomic type~\eqref{f11}. One of our main goals is Table~\ref{tab5} (analogous to Table~\ref{jinv} in the inner case) describing the parameters and restrictions on the $J$-invariant for groups of outer type.

Finally, we provide formulae for exact values of the entries of the $J$-invariant of degrees $1$ and $2$ for absolutely simple groups and relate these values to the Tits algebras of~$E$. More precisely, it is well known that the groups of type $^2\A_n$ correspond to central simple algebras over a quadratic field extension of the base field together with a unitary involution, and the groups of type $^2\D_n$ correspond to central simple algebras over the base field with an orthogonal involution of a non-trivial discriminant. Splitting the respective central simple algebra reduces these groups to Hermitian and resp. to quadratic forms. The latter algebraic structures are considered to be ``less complex'' compared to the situation when the algebra is present.

We show that in both cases the entries of the $J$-invariant of degrees bigger than $2$ do not change when we generically split the respective central simple algebras.
Moreover, we explicitly compute the degree one and degree two entries of the $J$-invariant. In the case $^2\A_n$ the degree one entry is related to the discriminant algebra of the respective unitary involution and the degree two entry is related to the underlying central simple algebra. In the case $^2\D_n$ we have an opposite situation, when the degree one entry of the $J$-invariant is related to the underlying central simple algebra and the degree two entry is related to the Clifford algebra of the orthogonal involution
 (see Proposition~\ref{prounit}, Corollary~\ref{cor67}, Proposition~\ref{prop69}, Corollary~\ref{cor718}). Analogous statements for groups of inner type were obtained in \cite{Zh} by the second author.

In our proofs we use inter alia the Rost invariant, Karpenko's upper motives, Fino's computations of the $J$-invariant of Hermitian forms, and Vishik's computations of the Chow rings of orthogonal Grassmannians.

\section{Preliminaries on linear algebraic groups of outer type}\label{sec2}
Let $G$ be a semisimple linear algebraic group over $F$.
There is the induced action of the absolute Galois group of $F$ on the Dynkin diagram of $G$ and on the Weyl group $W$ of $G$, which
is called the $*$-action  (see \cite{Bor}, \cite{KMRT}, \cite{Spr}).

A {\it twisted flag variety} of $G$ over $F$ is the variety of parabolic subgroups of $G$ of some fixed type.
The type of a parabolic subgroup of $G$ is a subset of the set of vertices of the Dynkin diagram of $G$
invariant under the $*$-action,
and we write $X_\Psi$ for the variety of parabolic subgroups of $G$ of type $\Psi$. Under this identification
the variety of Borel subgroups has type $\emptyset$.

We say that $G$ is a group of {\it inner type}, if
the group $G$ is a twisted form of a split
group $G'$ over $F$ by an element $\xi$ in the image of $H^1(F,G'/C)\to H^1(F,\Aut(G'))$, where $C$
is the center of $G'$ (the map is induced by taking the inner automorphisms) or equivalently if the $*$-action of
the absolute Galois group of $F$ on the Dynkin diagram of $G$ is trivial. We say that $G$ is a group
of {\it outer type}, if $G$ is not of inner type.

In this article we often consider a minimal field extension $K$ of $F$ such that the group $G_K$ is of inner type. This field is a subfield of the separable closure $F_{\mathrm{sep}}$ of $F$ corresponding to the kernel of the $*$-action on $G$.

Let $X$ be a twisted flag variety over $F$. By $\overline X$ we denote the variety $X_L$ over a splitting field
$L$ of the group $G$. The Chow ring $\CH(\overline X)$ does not depend on the choice of $L$ and,
therefore, we do not specify the splitting fields in the formulae below. Note also that
the variety $\overline X$ is irreducible.

The absolute Galois group $\Gamma$ of $F$
acts on $\CH(\overline X)$ via the $*$-action. Namely, $\CH(\overline X)$ is a free abelian group with a basis
given by the Schubert subvarieties. The Schubert subvarieties are parametrized by the left cosets $W/W_{\Psi}$,
where $W$ is the Weyl group of $G$ generated by the simple reflections $s_i$ and $W_{\Psi}$ is the subgroup of $W$
generated by $s_i$ with $i\in\Psi$, where $\Psi$ is the type of $X$. We identify these left cosets
with their (unique) representatives of the minimal length. Then for $\sigma\in\Gamma$ the Schubert variety corresponding to
$s_{i_1}\cdots s_{i_n}$ maps to $(\sigma*s_{i_1})\cdots(\sigma*s_{i_n})$.

We say that a twisted flag $G$-variety $X$ is generically split (resp. generically quasi-split), if the group $G_{F(X)}$ is split (resp. quasi-split).

\section{Preliminaries on the category of Chow motives}\label{sec3}
Denote by $\CM(F,\ff_p)$
the category of the effective Grothendieck's Chow motives over a field $F$ with coefficients
$\ff_p$ for a fixed prime number $p$ (see \cite{Ma68}, \cite{EKM}).

For a smooth projective variety $X$ over $F$ we denote by $M(X)$ the motive of $X$ in this category.
We consider the Chow ring $\CH(X)$ of $X$ modulo rational equivalence (see \cite{Ful}) and we write $\Ch(X)$ for $\CH(X)\otimes\ff_p$.

For a motive $M$ (resp. for a variety $X$) over a field $F$ and a field extension $E/F$ we denote by $M_E$
(resp. by $X_E$) the extension of scalars. The motive of the projective line $\mathbb{P}^1$  over $F$ decomposes into a direct sum $\ff_p\oplus\ff_p(1)$, where $\ff_p$ denotes the motive of $\Spec F$ and $\ff_p(1)$ is the complementary direct summand called the Tate motive.
By $M(n):=M\otimes\ff_p(1)^{\otimes n}$ we denote the Tate twist of $M$ by $n$.

Every motive $M$ is determined by a smooth projective variety $X$ and a projector
${\pi\in\Ch(X\times X)}$.
A motivic decomposition of a motive $M$ is a decomposition into a direct sum.
Motivic decompositions of a motive $M=(X,\pi)$ correspond to decompositions of the projector $\pi$
into a sum of (pairwise) orthogonal projectors.
The {\it realization} (resp. the {\it co-realization}) of a cycle $\beta\in\Ch(X\times X)$ is the map
$\beta_*\colon\Ch(X)\to \Ch(X)$ (resp. $\beta^*\colon\Ch(X)\to \Ch(X)$) given by $x\mapsto (p_2)_*(\pi\cdot p_1^*(x))$ (resp. $x\mapsto (p_1)_*(\pi\cdot p_2^*(x))$), where $p_i\colon X\times X\to X$ is the projection to the $i$-th factor ($i=1,2$), $p_i^*$ is the pullback, $(p_i)_*$ is the push-forward and $\cdot$ is the intersection product.
By $\circ$ we denote the composition product.

For an equidimensional smooth variety $X$ and for an element $${\alpha\in\Ch_l(X)=\Ch^{\dim X-l}(X)}$$
we write $\dim\alpha=l$ and $\codim\alpha=\dim X-l$.

For a field extension $E/F$ and a variety $Y$ over $F$ we
call a cycle $\rho\in\Ch(Y_E)$ {\it rational}, if it is defined over $F$, i.e. lies in the image of the
restriction homomorphism $$\res_{E/F}\colon\Ch(Y)\to\Ch(Y_E).$$ We say that $\rho$ is $F$-rational, if we want to stress $F$.

Let $X$ and $Y$ be smooth projective equidimensional varieties over $F$. A {\it correspondence} $\alpha\colon X\rightsquigarrow Y$
of degree $l$ is an element in $\Ch_{\dim X+l}(X\times Y)$. The {\it multiplicity} of $\alpha$ is an element $\mult(\alpha)\in\ff_p$
such that $\pr_*(\alpha)=\mult(\alpha)\cdot [X]$, where $\pr_*$ is the push-forward of the projection
$\pr\colon X\times Y\to X$. By $\Corr(X,Y)$ we denote the group of all correspondences from $X$ to $Y$ over $F$ of
degree $0$. We also write $\End_{\CM}(X)=\Corr(X,X)$. This is the group of endomorphisms of
$M(X)$ in the category of motives.

If $X$ is a twisted flag variety under a semisimple algebraic group $G$, then following Karpenko we consider the {\it upper motive} $\U(X)$
of $X$. By definition $\U(X)$ is an indecomposable direct summand of $M(X)$ such that $\Ch^0(\U(X))\ne 0$.

If $N$ is a direct summand of the motive of $X$, then we consider the Poincar\'e polynomial $P(N,t)$ of $N$ defined as $\sum_{i\ge 0}\dim\Ch^i(\overline N)\, t^i$, where $\overline N$ denotes the motive $N$ over a splitting field of $G$.

For a variety $X$ over a finite separable field extension $K/F$ we consider the {\it corestriction} $\cores_{K/F}(X)$ of $X$. This is a variety over $F$
obtained from $X$ via the composite map $$X\to \Spec K\to\Spec F.$$
As in \cite{Ka10} we consider the corestrictions of motives and correspondences. For a motive $M$ over $K$ and for
a correspondence $\alpha\colon X\rightsquigarrow Y$ over $K$ (where $X$ and $Y$ are smooth projective equidimensional
varieties over $K$) we denote the corestriction of $M$
by $\cores_{K/F}(M)$ and the corestriction of $\alpha$ by $\cores\alpha\colon\cores_{K/F}X\rightsquigarrow \cores_{K/F}Y$.
Corestriction should not be confused with the Weil restriction.

\begin{example}
Let $F=\mathbb{R}$, $E=\mathbb{C}$, and $X=\Spec E[x,y]/(ix^2+y^2+1)$.
Then $$\cores_{E/F}(X)=\Spec F[t,x,y]/(t^2+1,tx^2+y^2+1)$$ and $(\cores_{E/F}(X))_E=\Spec E[x,y]/(ix^2+y^2+1)\sqcup\Spec E[x,y]/(-ix^2+y^2+1)$.
\end{example}

Finally, we summarize some properties of the Chow motives of Weil restrictions.

Let $L/F$ be a finite separable field extension, let $K$ be a finite Galois field extension containing $L$ and let $M$ be a Chow motive over $L$. In \cite{Ka00,Ka12b,Ka15} Karpenko considers the Weil restriction functor $R_{L/F}$ for motives and, in particular, the motive $R_{L/F}(M)$ over $F$. By \cite[Lemma~2.1]{Ka15} if $M_i$ are motives over $L$, then
\begin{align}\label{weil1}
R_{L/F}(M_1\oplus\ldots\oplus M_s)=R_{L/F}(M_1)\oplus\ldots\oplus R_{L/F}(M_s)\oplus N,
\end{align}
where $N$ is a direct sum of corestrictions to $F$ of motives over fields $E$ with $F\subsetneq E \subset K$.

Moreover, for motives $M$ and $N$ over $L$ we have by \cite[Lemma~2.1]{Ka12b}
\begin{align}
R_{L/F}(M\otimes N)=R_{L/F}(M)\otimes R_{L/F}(N)
\end{align} and for a Tate motive $\ff_p(i)$ we have
\begin{align}\label{weil3}
R_{L/F}(\ff_p(i))=\ff_p(mi)
\end{align}
with $m=[L:F]$ (here the first Tate motive is over $L$ and the second Tate motive is over~$F$).

Besides, for a smooth variety $X$ over $L$ we will use in Section~\ref{sec8} a natural map $$R\colon\CH^*(X)\to\CH^{m*}(R_{L/F}(X))$$ defined in \cite{Ka00}.

\section{Background on the $J$-invariant for groups of inner type}
Let $G$ be a split semisimple algebraic group over a field $F$, let $T$ be a split maximal torus of $G$, and let $B$ be a Borel subgroup of $G$ containing $T$.

\begin{ntt}\label{sec31}
The variety $G/B$ is cellular and, therefore, the Chow group $\CH^*(G/B)$ is a free $\zz$-module. Its free generators can be parametrized by the elements of the Weyl group $W$ of $G$. More precisely, for each $w\in W$ one
associates with it the class of the Schubert subvariety
$Z_w=[\overline{Bw_0wB/B}]\in\CH^{l(w)}(G/B)$, where $w_0$ is the longest element of $W$ and $l(w)$ denotes the length of $w$.

Let $\B B$ denote the classifying space of $B$ (see \cite{EG98}, \cite{To99}). By \cite[(4.1)]{Gr58} there is a characteristic map
\begin{align}\label{f42}
c\colon\CH^*(\B B)\to\CH^*(G/B)
\end{align}
which is a ring homomorphism. Besides, the pullback of the canonical projection ${G\to G/B}$ induces a ring homomorphism $\pi\colon \CH^*(G/B)\to \CH^*(G)$.

It follows from \cite[p.~21, Rem.~$2^\circ$]{Gr58} (see also \cite[Proposition~5.1]{GiZ12})
that the sequence
\begin{equation}\label{seqq}
\CH^*(\B B)\xrightarrow{c}\CH^*(G/B)\xrightarrow{\pi} \CH^*(G)
\end{equation}
of graded rings is right exact (i.e. $\pi$ is surjective and its kernel is the ideal of $\CH^*(G/B)$ generated by the elements of positive degrees in the image of $c$). Then the explicit combinatorial description of the map $c$ given in \cite{De74} allows to compute explicitly the ring structure of $\CH^*(G)$ (see also \cite[Section~2]{PS10}, \cite[Section~5]{GPS16}).

In fact, one can replace $B$ in sequence~\eqref{seqq} by any special parabolic subgroup $P$ of $G$ (see \cite[Lemma~7.1]{PS17}). Note also that by \cite[Theorem~6.4]{KM06} the image of $c$ coincides with the subring of rational cycles for the twisted form $_E(G/B)$ (or, more generally, $_E(G/P)$ for a special parabolic subgroup $P$) under a generic $G$-torsor $E$.
\end{ntt}

\begin{ntt}[$J$-invariant for Chow motives]\label{jinvold}
For a fixed prime $p$ we denote by $\Ch^*:=\CH^*\otimes\ff_p$ the Chow ring modulo $p$.
Let $G$ be a split semisimple algebraic group over a field $F$, $B$ a Borel subgroup of $G$ and $E$ a $G$-torsor over $F$. Then
\begin{equation}\label{pres}
\Ch^*(G)\simeq \ff_p[e_1,\ldots,e_r]/(e_1^{p^{k_1}},\ldots,e_r^{p^{k_r}})
\end{equation}
for some non-negative integers $r$, $k_i$ and with $\deg e_i=:d_i$. We assume that the sequence of $d_i$ is non-decreasing.

We introduce an order on the set of additive generators 
of $\Ch^*(G)$, i.e., on the monomials $e_1^{m_1}\ldots e_r^{m_r}$. 
To simplify the notation, we denote the monomial
$e_1^{m_1}\ldots e_r^{m_r}$ by $e^M$, where $M$ is an $r$-tuple
of integers $(m_1,\ldots,m_r)$. The codimension (in the Chow ring) of
$e^M$ is denoted by $|M|$. Observe that $|M|=\sum_{i=1}^rd_im_i$.

Given two $r$-tuples $M=(m_1,\ldots,m_r)$ and $N=(n_1,\ldots,n_r)$ we say
$e^M\le e^N$ (or equivalently $M\le N$) if either $|M|<|N|$, or $|M|=|N|$ and 
$m_i\le n_i$ for the greatest $i$ such that $m_i\ne n_i$.
This gives a well-ordering on the set of all monomials ($r$-tuples).

\begin{dfn}[{\cite[Definition~4.6]{PSZ08}}]\label{def71}
Denote as $\overline{\Ch}^*(G)$ the image of the composite map
$$
\Ch^*(E/B)\xrightarrow{\res} \Ch^*(G/B) \xrightarrow{\pi} \Ch^*(G),
$$
where $\pi$ is the pullback of the canonical projection $G\to G/B$ and $\res$ is the scalar extension to a splitting field of the torsor $E$.

For each $1\le i\le r$ set $j_i$ to be the smallest non-negative
integer such that the subring $\overline{\Ch}^*(G)$ contains an element $a$ 
with the greatest monomial $e_i^{p^{j_i}}$ 
with respect to the order on $\Ch^*(G)$ as above, i.e.,
of the form 
$$
a=e_i^{p^{j_i}}+\sum_{e^M\lneq e_i^{p^{j_i}}} c_M e^M, \quad c_M\in\ff_p.
$$
The $r$-tuple of integers $(j_1,\ldots,j_r)$ is called
the {\it $J$-invariant} of $E$ modulo $p$ and is denoted by $J(E)$ or $J_p(E)$.
Note that $j_i\le k_i$ for all $i$.
\end{dfn}

In fact, in this definition one can replace the Borel subgroup $B$ by any special parabolic subgroup $P$ of $G$.

By \cite{PSZ08} the Chow motive of $E/B$ with coefficients in $\ff_p$ decomposes into a direct sum of Tate twists of an indecomposable motive $R_p(E)$, and the Poincar\'e polynomial of $R_p(E)$ over a splitting field of $E$ equals
\begin{equation}\label{fpoin}
\prod_{i=1}^r\frac{t^{d_ip^{j_i}}-1}{t^{d_i}-1},
\end{equation}
where $(j_1,\ldots,j_r)$ is the $J$-invariant of $E$.

Moreover, if the degrees $d_i$ are pairwise distinct (this happens for example when $G$ is a simple group of type different from $\D_n$), then one can uniquely recover the $J$-invariant out of the Poincar\'e polynomial~\eqref{fpoin}.

We summarize the parameters of the $J$-invariant in Table~\ref{jinv}. The restrictions on $j_i$ in the last column of the table follow from \cite[Proposition~5.12]{Vi05}, \cite[Section~4]{PSZ08}, \cite[Corollaries~8.10 and 10.4]{GPS16}. Note that the most restrictions in the last column follow from the structure of the Steenrod algebra. The Steenrod operations modulo $p$ were constructed in algebraic geometry by Voevodsky in characteristic different from $p$  (see \cite{Vo03a}, \cite{Br99}) and by Primozic in characteristic $p$ (see \cite{Pr20}). Note also that there are further restrictions on the $J$-invariant not mentioned in the table (see e.g. \cite[Proposition~8.4]{PS21} for type $\E_8$).

By $p^s\parallel m$ we denote the highest power of $p$ which divides $m$.

\begin{center}
\begin{longtable}{l|l|l|l|l|l}
\caption{Parameters of the $J$-invariant (inner type)}\label{jinv}\\
Group $G$ & $p$ & $r$ & $d_i$, $i=1\ldots r$ & $k_i$, $i=1\ldots r$ & Restrictions on $j_i$\\
\hline
$\SL_n/\mu_m,\,m|n$       & $p|m$ & $1$           &  $1$  & $p^{k_1}\parallel n$ &   \\
$\PGSp_n,\,2|n$           & $2$       & $1$                       & $1$    & $2^{k_1}\parallel n$         &  \\
$\mathrm{O}^+_n$, $n>2$       & $2$       & $[\frac{n+1}{4}]$       & $2i-1$    & $[\log_2\frac{n-1}{2i-1}]$   & if $d_i + l = 2^sd_m$ and\\
&&&&& $2\nmid\binom{d_i}{l}$, then $j_m \le j_i+s$ \\
$\Spin_{2n}^{\pm},\,2\vert n$ & $2$       & $\frac{n}{2}$         & $1,\,i=1$  &$2^{k_1}\parallel n$          & the same restrictions\\
                              &           &                  & $2i-1,\,i\ge 2$  & $[\log_2\frac{2n-1}{2i-1}]$ & \\
$\Spin_n$, $n>2$              & $2$       & $[\frac{n-3}{4}]$         & $2i+1$  & $[\log_2\frac{n-1}{2i+1}]$  & the same restrictions\\
$\PGO_{2n}^+$, $n>1$            & $2$       & $[\frac{n+2}{2}]$      & $1,\,i=1$ & $2^{k_1}\parallel n$         & the same restrictions \\
                              &           &                 & $2i-3,\,i\ge 2$  &$[\log_2\frac{2n-1}{2i-3}]$   & assuming $i,m\ge 2$\\
$\G_2$              & $2$       & $1$     & $3$ & $1$                         & \\
$\F_4$              & $2$       & $1$     & $3$ & $1$                         & \\
$\E_6$              & $2$       & $1$     & $3$ & $1$                         & \\
$\F_4$       & $3$       & $1$            & $4$  & $1$                          & \\
$\E_6^{sc}$       & $3$       & $1$       & $4$  & $1$                          & \\
$\E_7$       & $3$       & $1$            & $4$  & $1$                          & \\
$\E_6^{ad}$   & $3$       & $2$                & $1,\,4$  & $2,\,1$                     & $|j_1-j_2|\le 1$\\
$\E_7^{sc}$                   & $2$       & $3$                  & $3,\,5,\,9$ & $1,\,1,\,1$    & $j_1\ge j_2\ge j_3$\\
$\E_7^{ad}$                   & $2$       & $4$                & $1,\,3,\,5,\,9$ & $1,\,1,\,1,\,1$           & $j_2\ge j_3\ge j_4$\\
$\E_8$                        & $2$       & $4$             & $3,\,5,\,9,\,15$  & $3,\,2,\,1,\,1$             & $j_1\ge j_2\ge j_3$,\\
&&&&& $j_1\le j_2+1$,\\
&&&&&$j_2\le j_3+1$.\\
$\E_8$                        & $3$       & $2$            & $4,\,10$           & $1,\,1$                    & $j_1\ge j_2$\\
$\E_8$                        & $5$       & $1$                        & $6$   & $1$                          &\\
\end{longtable}
\end{center}
\end{ntt}

\section{Normed Chow rings and motives}\label{sec5}

Let $F$ be a field, let $K/F$ be a Galois field extension and let $X$ be a smooth variety over $F$. Following \cite{Fi19} we define $$\CH_K(X)=\CH(X)/\Im\Big(\bigoplus_{F\subsetneq E\subset K}\CH(X_E)\to\CH(X)\Big),$$ where $\CH(X_E)\to\CH(X)$ is the pushforward of the natural morphism $X_E\to X$.

The set $\Im\Big(\bigoplus_{F\subsetneq E\subset K}\CH(X_E)\to\CH(X)\Big)$ is an ideal of the ring $\CH(X)$ called the {\it normed ideal} and we call the ring $\CH_K(X)$ the {\it normed Chow ring} of $X$ (with respect to $K/F$).

From now on we assume that $[K:F]=p^n$ for some prime number $p$ and for some $n\ge 1$, since otherwise the respective normed Chow rings are zero.
Note also that then $\CH_K(X)$ is an algebra over $\CH_K(\pt)=\ff_p$
where $\pt$ denotes $\Spec F$. Furthermore, for a smooth variety $Y$ over $E$ we have ${\CH_K(\cores_{E/F}(Y))=0}$ where ${F\subsetneq E\subset K}$.

\begin{rem}
If $X$ is a split flag variety over $F$ and $K$ is a quadratic field extension of $F$, then the ring $\CH^*_K(X)$ coincides with the ordinary modulo $2$ Chow ring  $\CH^*(X)\otimes\ff_2$.
\end{rem}

We define the category of normed Chow motives exactly in the same way as the category of classical Grothendieck's Chow motives over $F$ (see \cite{Ma68})
 but replacing the usual Chow rings $\CH$ by the normed Chow rings $\CH_K$.

Let $X$ be a quasi-split twisted flag variety homogeneous under a quasi-split group $G^{qs}$ over $F$ with a Borel subgroup $B^{qs}$ defined over $F$.
Then there is still a characteristic map $$c\colon\CH^*(\B B^{qs})\to \CH^*(G^{qs}/B^{qs}).$$
Moreover, by \cite[Theorem~6.4]{KM06} the image of $c$ coincides with the subring of rational cycles for the twisted form $_E(G^{qs}/B^{qs})$ under a generic $G^{qs}$-torsor $E$ over $F$.

Furthermore, we have for an adjoint or simply connected group $G^{qs}$ that $${\CH^*(\B B^{qs})=\CH^*(\B B)^\Gamma}$$ and $\CH^*(G^{qs}/B^{qs})=\CH^*(G/B)^\Gamma$, where $\Gamma$ denotes the absolute Galois group, and $B$ and $G$ denote $B^{qs}$ and $G^{qs}$ extended to a separable closure of $F$.

Let $K$ be a Galois field extension of $F$ of $p$-primary degree such
that $G^{qs}_K$ is a group of inner type.
Passing to the normed Chow rings we obtain a characteristic map
\begin{equation}\label{normedc}
c\colon\CH_K^*(\B B^{qs})\to\CH_K^*(G^{qs}/B^{qs}).
\end{equation}

The ring $\CH_K^*(G^{qs}/B^{qs})$ is identified with the subring of $\CH^*(G/B)\otimes\ff_p$
generated by the Schubert classes $Z_w$ with $w\in W$ which are stable under $\Gamma$. In other words, taking the normed Chow ring of a quasi-split projective homogeneous variety eliminates all Schubert cycles from the ordinary Chow ring of a split variety which are not Galois invariant.

Analogous to sequence~\eqref{seqq}
we consider the sequence
\begin{equation}\label{normedc2}
\CH_K^*(\B B^{qs})\xrightarrow{c}\CH_K^*(G^{qs}/B^{qs})\to C^*
\end{equation}
where the graded ring $C^*$ is given as the quotient of the ring $\CH_K^*(G^{qs}/B^{qs})$ modulo the ideal generated by the elements of positive degrees in the image of $c$. 
We will show later in Proposition~\ref{prop77} that
one can replace $B^{qs}$ in sequence~\eqref{normedc2} by any parabolic subgroup $P$ of $G^{qs}$ such that the variety $E/P$ is generically quasi-split for a generic $G^{qs}$-torsor $E$.

Let $X$ be a twisted flag variety over $F$ and let $K$ be a Galois field extension of $F$ of $p$-primary degree that that the respective group is a group of inner type over $K$.
We define the {\it normed upper motive} $\U_K(X)$ of $X$ exactly in the same way as its ordinary upper motive (see Section~\ref{sec3}), but replacing the ordinary Chow groups by the normed Chow groups, i.e., $\U_K(X)$ is an indecomposable direct summand of the normed motive $M_K(X)$ of $X$ such that $\CH_K^0(\U_K(X))\ne 0$.

Note that in the context of normed Chow rings we will denote by $\overline X$ (and, more generally, by $\overline N$ for a direct summand $N$ of the normed motive of $X$) the variety $X$ not over a splitting field of the respective group, but over the function field $L$ of a generically quasi-split variety (e.g., the variety of Borel subgroups). Over such fields the respective group is quasi-split, but the field $K$ remains a field, i.e., $K\otimes_F L$ is still a field.

Note that in the category of normed Chow motives
every direct summand of a {\it quasi-split} twisted flag variety is a direct sum of (normed) Tate motives. Therefore, for a normed motive $N$ as in the previous paragraph we can define its normed Poincar\'e polynomial $P(\CH_K^*(N),t)$ as $\sum_{i\ge 0}\dim\CH_K^i(\overline N)\, t^i$, where the motive $\overline N$ is considered over a quasi-splitting field as above. 

\section{Some numerology of algebraic groups}

In Table~\ref{tab2} below we list torsion primes for absolutely simple algebraic groups. For groups of inner type they are characterized by the property that the characteristic map $c$ from formula~\eqref{f42} is surjective modulo $q$, if $q$ is not a torsion prime, and for groups of outer type we define the torsion primes just as divisors of the degree of the minimal field extension making the group inner.
Moreover, we do not consider the situation of mixed primes, like $^6\D_4$, since as was mentioned in Section~\ref{sec5} the normed Chow rings vanish for quasi-split projective homogeneous varieties
of respective types.

\begin{center}
\begin{longtable}{c|c}
\caption{Torsion primes}\label{tab2}\\
\hline
$\A_m$ & $p\mid (m+1)$\\
$\B_m,\C_m,\D_m,\G_2$ & $2$\\
$\F_4,\E_6,\E_7$ & $2,3$\\
$\E_8$ & $2,3,5$\\
\hline
${}^2\A_m, {}^2\D_m, {}^2\E_6$ & $2$\\
${}^3\D_4$ & $3$
\end{longtable}
\end{center}

Table~\ref{tab3} below contains the well-known degrees of fundamental polynomial invariants (see \cite[Corollary~10.2.4 and Proposition~10.2.5]{Car}). These numbers appear in formulae for Poincar\'e polynomials for split flag varieties (see Proposition~\ref{solo} below).

\begin{center}
\begin{longtable}{c|c}
\caption{Degrees of fundamental polynomial invariants}\label{tab3}\\
Dynkin type & $e_i$\\
\hline
$\mathrm{A}_m$ & $2,3,\ldots,m+1$\\
$\mathrm{B}_m,\mathrm{C}_m$ & $2,4,\ldots, 2m$\\
$\mathrm{D}_m$ & $2,4,\ldots, 2m-2,m$\\
$\mathrm{E}_6$ & $2,5,6,8,9,12$\\
$\mathrm{E}_7$ & $2,6,8,10,12,14,18$\\
$\mathrm{E}_8$ & $2,8,12,14,18,20,24,30$\\
$\mathrm{F}_4$ & $2,6,8,12$\\
$\mathrm{G}_2$ & $2,6$\
\end{longtable}
\end{center}

For groups of outer type one can also construct a table containing analogous integers (cf. \cite[Theorem~14.3.2]{Car}). We summarize them in Table~\ref{tab4} below. Note that the numbers $e_i^+$ in the table are nothing else than the degrees of fundamental polynomial invariants for the respective folded root systems. Note also that the numbers from the last two columns of Table~\ref{tab4} taken together give precisely the degrees of fundamental polynomial invariants for respective groups of inner type from Table~\ref{tab3}.

\begin{center}
\begin{longtable}{c|c|c|c}
\caption{Degrees of fundamental polynomial invariants (outer type)}\label{tab4}\\
Dynkin type & Folded root system & $e^+_i$ & $e^-_i$ (resp. $e^{\omega}_i, e^{\omega^2}_i$ in \\
 &  &  & the trialitarian case)   \\
\hline
    ${}^2\mathrm{A}_m$ & $\mathrm{C}_{(m+1)/2}$ ($m$ odd),  & all even numbers   & all odd numbers \\
     & $\BC_{m/2}$ ($m$ even) & between $2$ and $m+1$ &  between $3$ and $m+1$\\
${}^2\mathrm{D}_m$ & $\B_{m-1}$ & $2,4,\ldots, 2m-2$ & $m$\\
${}^2\mathrm{E}_6$ & $\F_4$ & $2,6,8,12$ & $5,9$\\
${}^3\mathrm{D}_4$
& $\G_2$ & $2,6$ & $4,4$
\end{longtable}
\end{center}

Let $G$ be a simply-connected semisimple quasi-split algebraic group over $F$, let $K$ be a Galois field extension such that $G$ is of inner type over $K$, and let $P$ be a parabolic subgroup of $G$ defined over $F$. Then by \cite[Theorem~26.8]{KMRT} the group $G$ is isomorphic to a (finite) product of the Weil restrictions $R_{L_i/F}(G_i)$ where $L_i/F$ is a finite separable extension of $F$ and $G_i$ is an absolutely simple quasi-split group over $L_i$.

Furthermore, the variety $G/B$ is isomorphic to the product of varieties $R_{L_i/F}(G_i)/B_i$ where $B_i$ denotes the Borel subgroup of $R_{L_i/F}(G_i)$. Thus, the Poincar\'e polynomial $P(\CH^*_K(G/B),t)\in\Z[t]$ of $\CH^*_K(G/B)$ equals the product of the Poincar\'e polynomials $P(\CH^*_K(R_{L_i/F}(G_i)/B_i),t)$.

Considering the natural projection $G/B\to G/P$ with fiber $P/B$ we have that $$P(\CH^*_K(G/P),t)=P(\CH^*_K(G/B),t)/\CH(\CH^*_K(P/B),t).$$ Moreover, $P/B\simeq [L(P),L(P)]/B\cap [L(P),L(P)]$ where $L(P)$ denotes the Levi part of~$P$. Thus, combining the above observations to compute the Poincar\'e polynomial of $\CH_K^*(G/P)$
we can reduce to the case when the group $G$ is simple and $P$ is a Borel subgroup of $G$.

\begin{prop}[Extended Solomon theorem]\label{solo}
Let $G$ be a simple quasi-split group over~$F$, let $B$ be a Borel subgroup of $G$ defined over $F$ and let $K$ be a finite Galois field extension such that the group $G_K$ is of inner type.

a) If $G$ is a simple algebraic group of inner type over $F$, then the Poincar\'e polynomial $P(\CH^*(G/B),t)$ equals 
\begin{equation}\label{solom}
\prod \dfrac{t^{e_i}-1}{t-1}
\end{equation}
where $e_i$ are degrees of the fundamental polynomial invariants from Table~\ref{tab3}.

b) If $G$ is an absolutely simple group of outer type over $F$ different from ${}^3\D_4$ and ${}^6\D_4$, then the Poincar\'e polynomial of $\CH^*_K(G/B)$ equals
\begin{equation}\label{solom2}
\prod \dfrac{t^{e^+_i}-1}{t-1}\cdot \prod \dfrac{t^{e^-_i}+1}{t+1}
\end{equation}
where $e^+_i$ and $e^-_i$ are taken from Table~\ref{tab4}.

c) If $G={}^3\D_4$,
then the Poincar\'e polynomial of $\CH^*_K(G/B)$ equals $$\dfrac{t^2-1}{t-1}\cdot \dfrac{t^6-1}{t-1}\cdot \dfrac{t^4-\omega}{t-\omega}\cdot \dfrac{t^4-\omega^2}{t-\omega^2}=\dfrac{t^2-1}{t-1}\cdot \dfrac{t^6-1}{t-1}\cdot\dfrac{t^8+t^4+1}{t^2+t+1},$$
where $\omega=e^{2\pi i/3}$ (see Table~\ref{tab4}).

d) Let $G=R_{E/F}(G')$ for a non-trivial finite separable field extension $E/F$ and an absolutely simple algebraic quasi-split group $G'$ over $E$. Assume that $[K:F]$ is a power of a prime number $p$. Let $B'$ be a Borel subgroup of $G'$ defined over $E$ and $m=[E:F]$.

If $G'$ is split over $E$, then the Poincar\'e polynomial $P(\CH^*_K(G/B),t)$ is equal to the polynomial $P(\CH^*(G'/B'),t^m)$.

If $G'$ is quasi-split but not split over $E$, then the Poincar\'e polynomial $P(\CH^*_K(G/B),t)$ equals $P(\CH_K^*(G'/B'),t^m)$. 
\end{prop}
\begin{proof}
Note that the coefficient of the polynomial $P(\CH^*(G/B),t)$ of $t^u$ equals the cardinality of the set $\{w\in W\mid \ell(w)=u\}$ where $W$ is the Weyl group of $G$ and $\ell$ is the length function.
Thus, the classical Solomon theorem (see \cite[Section~9.4A]{Car}) precisely gives the desired formula~\eqref{solom}.

For part~b) note that the Poincar\'e polynomial of the normed Chow group is given by the same formulae as the usual one but with the Weyl group replaced by the Weyl group $W^1$ of the folded root system defined in \cite[Section~13.1]{Car} (see also \cite[Section~13.3.8]{Car}). Therefore, formula~\eqref{solom2} follows from \cite[Proposition~14.2.1]{Car} together with computations of all parameters of Proposition~14.2.1 given in \cite[Section~14.3]{Car}. 

Part~c) follows from \cite[Proposition~14.2.1]{Car} or can be computed directly using formulae for the $*$-action given in Section~\ref{sec2}.

Part~d) follows from the properties of the Weil restrictions summarized in Section~\ref{sec3} and the fact that the normed Chow ring eliminates all cycles from the ordinary Chow ring which are not Galois invariant (see Section~\ref{sec5}). Note also that $K$ contains $E$, since by the assumptions the group $G$ is of inner type over $K$.
\end{proof}

\begin{example}
Consider a simple quasi-split group $G$ of type $\D_n$, $n\ge 3$, over $F$ and the maximal parabolic subgroup $P_1=P_{\{2,\ldots,n\}}$ of type $\{2,\ldots,n\}$ (the enumeration of simple roots follows Bourbaki). We assume for simplicity that the characteristic of $F$ is different from $2$ (though the formulae for the Poincar\'e polynomials below hold in every characteristic).

Then the group $G$ is of even orthogonal type and the variety $Q=G/P_1$ is a (quasi-split) projective quadric of dimension $2n-2$.

Assume first that the group $G$ is of inner type. This means that the discriminant of the quadric $Q$ is trivial. Then $$P(\CH^*(G/P_1),t)=\frac{(t^{n-1} + 1)(t^n - 1)}{t - 1}=1+t+\ldots+t^{n-2}+2t^{n-1}+t^{n}+\ldots+t^{2n-2}.$$

Let now $G$ be a group of outer type. Let $K$ be the minimal field extension of $F$ such that $G_K$ is a group of inner type, i.e., $K$ is a quadratic extension $K=F(\sqrt{\disc Q})$.
Then $$P(\CH_K^*(G/P_1),t)=\frac{(t^{n-1}-1)(t^n+1)}{t-1}=1+t+\ldots+t^{n-2}+t^n+\ldots+t^{2n-2}.$$
\end{example}

\section{Motives of generically quasi-split varieties}\label{sec7}

We define the $J$-invariant for groups of outer type exactly in the same way as for the inner type, but with Chow rings replaced by the normed Chow rings. Nevertheless, we need to make several adjustments.

First of all, we need to replace $\Ch^*(G)$ in Definition~\ref{def71} by the ring $C^*$ from sequence~\eqref{normedc2}. Moreover, we need to know that, as in the inner case, this ring has presentation~\eqref{pres} and we need to compute the parameters $r$, $d_i$, and $k_i$ for this presentation. Once this is done, the same proof as in the inner case works, and we obtain an analogous motivic decomposition for normed motives as in the inner case (see Theorem~\ref{prop612} below).

Let $Y$ be a generically quasi-split twisted flag variety over $F$. Let $K$ be a Galois field extension of $F$ of $p$-primary degree such that the respective group is a group of inner type over $K$.
We denote by $\overline{\CH}_K^*(Y)$ the subring of rational cycles on $Y$ over $F(Y)$. In order to describe the ring $C^*$ we need to investigate the structure of the ring $\CH_K(\overline{Y})$ as a $\overline{\CH}_K(Y)$-module, where $\overline Y$ stands for $Y_{F(Y)}$. Note that since $Y$ is generically quasi-split, the normed motive of $Y$ over $F(Y)$ is a direct sum of Tate motives.

To begin with, we have the following well-known fact about the Chow group of the upper motive of $Y$.
\begin{lem}
\label{ChU}
$\overline{\CH}_K(\U_K(Y)) = \ff_p \cdot 1$.
\end{lem}
\begin{proof}
Since $Y$ is generically quasi-split, every rational cycle in ${\CH}_K(\U_K(Y))$ is a generic point of some direct summand of $\U_K(Y)$ (the proof is the same as for the inner case; see \cite[Corollary~4.11]{GPS16}).

The lemma now follows, since the motive $\U_K(Y)$ is indecomposable.
\end{proof}

We fix now a projector $\pi \in \End(M_K(Y))$ of the normed motive of $Y$ such that $(Y, \pi) \simeq \U_K(Y)$. Recall that
$\End(Y, \pi) = \pi \circ \End(M_K(Y)) \circ \pi \subset \CH_K^{\dim Y}(Y \times Y) $. Denote by 
$\overline{\End}(Y,\pi)$ the image of the restriction map $\End(Y,\pi) \rightarrow \End(\overline{Y,\pi})$ and write $\overline{\pi}$ for the image of $\pi$ in $\overline{\End}(Y,\pi)$.

For cycles in the normed Chow groups we consider their multiplicities $\mult$ defined exactly in the same way as in the case of ordinary Chow groups (see Section~\ref{sec3}).

\begin{lem}
\label{End}
$\overline{\End}(Y,\pi) = \ff_p \cdot \overline{\pi}$. 
\end{lem}
\begin{proof}
Note that, since $(Y, \pi) \simeq \U_K(Y)$, we have  $\mult \overline{\pi}=1$. Let $\overline{\alpha} \in \overline{\End}(Y,\pi)$ and let $\alpha$ be its preimage in $\End(Y,\pi)$. Replacing if necessary $\alpha$ by $\alpha - \lambda \cdot \pi$ with $\lambda \in \ff_p$, we can assume that $\mult \alpha =0$ and our goal is to show that $\alpha = 0$.

We fix a homogeneous $\ff_p$-basis $\mathcal{B}$ of $\CH_K(\overline{Y})$ with its dual basis $\mathcal{B}^*$ and  write $${\overline{\alpha} = \Sum_{i\in I}a_i\times b_i}$$ for some index set $I$, where $b_i \in \mathcal{B}$. We choose $a_l$, $l \in I$,  with the minimal codimension among all elements $a_i, i \in I$. Note that, since  $\mult \alpha =0$, we have ${\codim a_l > 0}$. Consider the cycle $b_l^* \in \mathcal{B}^*$ dual to $b_l$. That is, for every $i\in I$ the degree $\deg(b_ib_l^*)$ is equal to $1$ if $i=l$ and is equal to $0$ otherwise.   

Let $\beta \in \CH_K( Y \times Y)$  be a preimage of $b_l^*$ under the flat pull-back
$$\CH_K(Y \times Y) \, \, -\!
\!\! \twoheadrightarrow \CH_K(Y_{F(Y)})$$
\noindent with respect to the morphism 
$ Y_{F(Y)}=\Spec F(Y) \times Y \rightarrow Y\times Y$ given by the generic point of $Y$.

Since $\overline{\beta}= 1 \times b_l^* + \dots $, where $`` \dots "$ stands for a linear combination of only those terms whose first factor has codimension $> 0$, we have
$$\overline{(\pr_1)_*(\alpha \beta)} =  (\pr_1)_*(\overline{\alpha\beta}) = (\pr_1)_* (a_l \times [\pt] +  \dots) =a_l \, ,$$
\noindent where $[\pt]$ denotes the class of a rational point and $`` \dots "$ stands for a linear combination of only those terms whose second factor has dimension $> 0$. It follows that $a_l$ is $F$-rational. Since $a_l \in \CH_K(\overline{Y, \pi})$ and $\codim a_l >0$, by Lemma~\ref{ChU} we have $a_l=0$ and, therefore, $\alpha =0$.
\end{proof}

\begin{cor}
\label{mult}
Let $\overline{\alpha} \in \overline{\End}(Y,\pi)$. Then 
$$\mult \overline{\alpha} = 1 \, \, \Longrightarrow   \, \,  \overline{\alpha} = \overline{\pi} \quad \text{and} \quad \mult \overline{\alpha} = 0 \, \, \Longrightarrow  \, \,\overline{\alpha} = 0 \, . $$
\end{cor}

\begin{lem}
\label{R-basis}
The ring $\CH_K(\overline{Y})$ is a free $\overline{\CH}_K(Y)$-module. Moreover, for a basis of this module one can take any $\ff_p$-basis of $\CH_K(\overline{\U_K(Y)})$. 
\end{lem}
\begin{proof}
Let $b_1, \ldots, b_r$ be a homogeneous $\ff_p$-basis of $\overline{\CH}_K(Y)$  and let ${b^*_1,\ldots, b^*_r \in \CH_K(\overline{Y})}$ be a family of cycles dual to this basis with respect to the bilinear form
\begin{align*}
\CH_K(\overline{Y})\times \CH_K(\overline{Y}) &\rightarrow \ff_p\\
(x,y) &\mapsto \deg(x\cdot y)
\end{align*}

Let $\pi \in \End(M_K(Y))$ be a projector as above, i.e., $(Y, \pi) \simeq \U_K(Y)$. We claim that $\CH_K(\overline{Y}) = \Oplus_{i=1}^r b_i\CH_K(\overline{Y,\pi})$. In order to show this, we will construct $F$-rational pairwise orthogonal projectors $\overline{\pi_i}$, $i=1,\dots,r$, as follows.

Let $i \in [1,r]$ and let $x_i \in \CH_K(Y)$ be a preimage of $b_i \in \CH_K(\overline{Y})$. We construct two correspondences $\alpha_i, \beta_i \in \CH_K(Y\times Y)$ as follows. We set $\beta_i = (x_i \times 1)\cdot \pi$ and define $\alpha_i$ as a cycle whose image under the surjection
$\CH_K(Y \times Y) \, \, -\!
\!\! \twoheadrightarrow \CH_K(\Spec F(Y) \times Y) \simeq \CH_K(\overline{Y}) $
is $b_i^*$. Note that by construction of $\alpha_i$ and $\beta_i$ we have $\mult (\beta_i\circ \alpha_i) =1$ and, hence, $\mult (\pi \circ \beta_i\circ \alpha_i \circ \pi) =1$.
By Corollary~\ref{mult} we obtain $ \overline{\pi \circ \beta_i\circ \alpha_i \circ \pi} = \overline{\pi}$. It follows that the correspondence $\overline{\alpha_i \circ \pi \circ  \beta_i} \in \CH_K(\overline{Y}\times \overline{Y})$ is a projector, which we denote by $\overline{\pi}_i$. Moreover, $\overline\pi_i$ are rational, since so are the cycles $\overline\alpha_i$, $\overline\pi$ and $\overline\beta_i$.
By construction of the correspondences $\alpha_i$ and $\beta_j$, we have $\mult (\beta_j\circ \alpha_i) =0$ if $i \neq j$. Thus, it follows from Corollary~\ref{mult} that the projectors $\overline{\pi}_1, \dots , \overline{\pi}_r$ are pairwise orthogonal.

Moreover, we have $\overline{\pi}_1 + \dots + \overline{\pi}_r = \Delta_{\overline{Y}}$, where $\Delta_{\overline{Y}}$ is the diagonal class in 
 $\CH_K(\overline{Y}\times \overline{Y} )$. Indeed, if the difference 
  $\Delta_{\overline{Y}} - \overline{\pi}_1 - \dots - \overline{\pi}_r $
is non-zero, then we can use the same argument as in the proof of Lemma~\ref{End} to produce an $F$-rational cycle in $\CH_K(\overline{Y})$ which is not a linear combination of $b_i$, $i=1, \ldots, r$.

The above decomposition of the diagonal class $\Delta_{\overline{Y}}$ implies the
decomposition $${\CH_K(\overline{Y}) = \Oplus_{i=1}^r \CH_K(\overline{Y},\overline{\pi}_i)}.$$ We claim that $\CH_K(\overline{Y},\overline{\pi}_i) =   b_i \CH_K(\overline{Y,\pi})$. Since ${\overline{\pi}^*_i = \overline{\beta}_i^* \circ \overline{\pi}^* \circ \overline{\alpha}_i^*}$, where the upper star denotes the co-realization (see Section~\ref{sec3}), we have
$$\CH_K(\overline{Y},\overline{\pi}_i) =\overline{\pi}^*_i \CH_K(\overline{Y}) \subset \overline{\beta}_i^* \CH_K(\overline{Y}) \subset b_i \CH_K(\overline{Y,\pi}) \, ,$$
where the second inclusion follows from the equality $\overline{\beta}_i = (b_i\times 1)\overline{\pi}$. Moreover, we have an isomorphism of split motives
$(\overline{Y},\overline{\pi}_i) \simeq \overline{(Y,\pi)}(k)$ with $k = \codim b_i$. It follows that
$\dim_{\ff_p} \CH_K(\overline{Y},\overline{\pi}_i) =\dim_{\ff_p} \CH_K(\overline{Y,\pi})$. Therefore, the inclusion $${\CH_K(\overline{Y},\overline{\pi}_i) \subset b_i \CH_K(\overline{Y,\pi})}$$ is an equality and we have $\CH_K(\overline{Y}) = \Oplus_{i=1}^r b_i\CH_K(\overline{Y,\pi})$.
\end{proof}

\begin{rem}
The same proof shows that Lemma~\ref{R-basis} holds also for groups of inner type. In particular, it provides a geometric proof of \cite[page~71, (2)]{Kac85} (cf. also \cite[Lemma~5.5]{PS21}).
\end{rem}

\begin{rem}\label{rem75}
Note that the projectors $\overline\pi_i$ in the proof of the above lemma are $F$-rational. Therefore, the decomposition of the diagonal class  $\Delta_{\overline{Y}} = \overline{\pi}_1 + \dots + \overline{\pi}_r$ provides the motivic decomposition  $M_K(Y) = \Oplus_{i=1}^r \U_K(Y)(\codim b_i)$.
\end{rem}

\begin{prop}
\label{polynom}
Let $G$ be a semisimple algebraic group over $F$ and let $K$ be a Galois field extension such that $G_K$ is of inner type.
Let $Y$ be a generically quasi-split $G$-variety. Let $U$ be the normed upper motive of $Y$.
Then the following polynomial identity holds
$$P(\CH^*_K(Y),t) = P(\CH^*_K(U),t) \cdot P(\overline{\CH}_K^*(Y), t) \, .$$
\end{prop}

\begin{proof}
Since $Y$ is generically quasi-split, the normed motive of $Y$ decomposes into a direct sum of the motive $U$ with some Tate twists
$$M_K(Y) = \bigoplus_{i\geq 0} \, \, U(i)^{\oplus a_i} \, , $$
\noindent where $a_i \geq 0$ is the number of the summands $U(i)$ in the above decomposition (see Remark~\ref{rem75}).
Taking the Poincar\'e polynomials of both sides of the above motivic decomposition we obtain the required polynomial identity.
\end{proof}

\begin{prop}\label{prop77}
The ring $C^*$ does not depend on the choice of a parabolic subgroup $P$ of a quasi-split semisimple algebraic group $G$ such that the variety $E/P$ is generically quasi-split for a generic $G$-torsor $E$.
\end{prop}
\begin{proof}
We will show a more general statement.

Let $X$ and $Y$ be two generically quasi-split twisted flag varieties over $F$ (we do not assume that they are homogeneous under the same group) such that their normed upper motives $\U_K(X)$ and $\U_K(Y)$ are isomorphic.

We consider the subrings $\overline{\CH}^*_K(X)$ and $\overline{\CH}^*_K(Y)$ of rational cycles in $\CH^*_K(\overline X)$ and $\CH^*_K(\overline Y)$ and the ideals $I_X$ and $I_Y$ generated by the elements of these subrings of positive codimensions.

The ring $C^*$ is just the ring $\CH^*_K(\overline X)/I_X$ resp. $\CH^*_K(\overline Y)/I_Y$ in the case when the twisted flag varieties $X$ and $Y$ are generic. Note that we do not use this in the proof. 

It is sufficient to prove that the respective factor-rings are isomorphic for the variety $X$ and for the product $X\times Y$ instead of $Y$. We have a natural projection $\pr\colon X\times Y\to X$ and its pullback $\pr^*$ induces a commutative diagram of rings:
$$\xymatrix{
\overline{\CH}^*_K(X\times Y) \ar[r] & \CH^*_K(\overline X\times \overline Y)  \ar[r]
&  \CH^*_K(\overline X\times \overline Y)/I_{X\times Y} \\
\overline{\CH}^*_K(X) \ar[r] \ar[u]& \CH^*_K(\overline X) \ar[r] \ar[u]& \CH^*_K(\overline X)/I_{X} \ar[u]}$$

According to \cite[Corollary~2.15, Proof of Lemma~2.14]{Ka10} we can choose the projectors $\pi \in \End(M_K(X))$ and $\pi' \in  \End(M_K(X\times Y))$ representing respectively the upper motives $\U_K(X)$ and $\U_K(X\times Y)$ in such a way that the pullback $\pr^*\colon \CH^*_K(X) \rightarrow \CH^*_K(X\times Y)$ induces an isomorphism $\CH^*_K(\overline{X}, \overline{\pi}) \simeq   \CH^*_K(\overline{X}\times \overline{Y}, \overline{\pi}')$. Then we have the following commutative diagram of $\ff_p$-vector spaces:

$$\xymatrix{
\CH^*_K(\overline{X}\times \overline{Y}, \overline{\pi}') \ar[r] & \CH^*_K(\overline{X}\times \overline{Y})  \ar[r]
&  \CH^*_K(\overline{X}\times \overline{Y})/I_{X\times Y} \\
\CH^*_K(\overline{X}, \overline{\pi}) \ar[r] \ar[u]& \CH^*_K(\overline{X}) \ar[r] \ar[u]& \CH^*_K(\overline{X})/I_{X} \ar[u]} $$

By Lemma~\ref{R-basis} the compositions in the first and second rows of the above diagram are isomorphisms. Hence,  $\CH^*_K(\overline X)/I_X$ and $\CH^*_K(\overline X\times \overline Y)/I_{X \times Y}$ are isomorphic as rings.
\end{proof}

\section{$J$-invariant for groups of outer type}\label{sec8}

{\bf Exceptional types.}

We consider exceptional simple groups first, i.e., groups of type $^2\E_6$ ($p=2$) and $^3\D_4$ ($p=3$). In these cases one can compute $C^*$ directly using a combinatorial definition of the homomorphism $c$.

Namely, we have $C^*=\ff_2[e_3,e_5,e_9]/(e_3^2,e_5^2,e_9^2)$ for type $^2\E_6$ and $C^*=\ff_3[e_4]/(e_4^3)$ for type $^3\D_4$. This information is sufficient to compute a complete motivic decomposition of the variety of Borel subgroups $E/B$ for these two Dynkin types.

On the other hand, there is an alternative approach to compute the motive of $E/B$ for types $^2\E_6$ and $^3\D_4$ which we present now.

If $G$ an absolutely simple simply-connected algebraic group over $F$, then there is the Rost invariant $R_G\colon H^1(-,G)\to H^3(-,\qq/\Z(2))$ (see \cite[\S31B]{KMRT}, \cite{GMS03}).

Furthermore, with an embedding of absolutely simple simply-connected algebraic groups $H\subset G$ one can associate a positive integer $n$ called the Rost multiplier (see \cite[Section~2.1]{Ga01}) such that the following diagram commutes

$$\xymatrix{
H^1(F,H) \ar[r]^-{R_H}\ar[d] & H^3(F,\qq/\Z(2)) \ar[d]^-{n\cdot} \\
           H^1(F,G)\ar[r]^-{R_G}  & H^3(F,\qq/\Z(2)).
           }$$
Note that the Rost multiplier is a combinatorial invariant and can be computed under the assumption that the base field is algebraically closed.

We will be interested in the embeddings of quasi-split simply-connected algebraic groups ${{}^3\D_4\to\F_4}$ (see \cite[Corollary~38.7]{KMRT}) and ${{}^2\E_6\to\E_7}$ (see \cite[Proposition~3.6]{Ga01}). Both these embeddings have Rost multiplier $1$. For an algebraic group $G$ over $F$ we identify $H^1(F,G)$ with the set of $G$-torsors over $F$.

For a ${}^3\D_4$-torsor $E\in H^1(F,{}^3\D_4)$ (resp. for an ${}^2\E_6$-torsor $E\in H^1(F,{}^2\E_6)$) we consider its image ${\widetilde E\in H^1(F,\F_4)}$ (resp. ${\widetilde E\in H^1(F,\E_7)}$).

We denote by $B$ the Borel subgroup of the group ${}^3\D_4$ (resp. of ${}^2\E_6)$) and by $\widetilde B$ the Borel subgroup of $\F_4$ (resp. of $\E_7$).

\begin{prop}\label{prop65}
 In the above notation the upper motives of $E/B$ and of $\widetilde E/\widetilde B$ in the category of ordinary Chow motives modulo a prime number $p$ are isomorphic.
\end{prop}
\begin{proof}
It suffices to prove that the twisted flag varieties $E/B$ and $\widetilde E/\widetilde B$ have zero-cycles of degree coprime to $p$ over the function fields of each other.

By \cite[Theorem~0.5]{Ga01} and \cite[Theorem~6.1]{Ch03} the kernel of the Rost invariant is trivial for groups of type $^3\D_4$, $\F_4$, $^2\E_6$, and $\E_7$. Since the Rost multiplier for embeddings of the respected pairs of groups ${{}^3\D_4\to\F_4}$  and ${{}^2\E_6\to\E_7}$ equals $1$, it follows that the Rost invariants of $E$ and $\widetilde E$ are equal and, therefore, the varieties $E/B$ and $\widetilde E/\widetilde B$ are quasi-split over the same fields.
\end{proof}

In particular, Proposition~\ref{prop65} immediately implies that in the above notation the $J$-invariants of $E$ and of $\widetilde E$ coincide and this gives the respective two lines for $^3\D_4$ and $^2\E_6$ in Table~\ref{tab5} below. Note that Proposition~\ref{prop65} holds for simply-connected groups, but we can put adjoint groups in Table~\ref{tab5}, since the index of the weight lattice in the root lattice is coprime to $p$ (with $p=3$ for $^3\D_4$ and $p=2$ for $^2\E_6$).\\

Next we consider the remaining two cases of absolutely simple groups of outer type, namely types $^2\A_n$ and $^2\D_n$. The arguments in these two cases are slightly different and, therefore, we study these cases separately.

{\bf Unitary groups.}

We start with the unitary case $^2\A_n$.
Groups of type $^2\A_n$ correspond to central simple algebras $B$ over a quadratic field extension of $F$ of degree $n+1$ with an involution of the second kind $\tau$. Firstly, we consider the case when the algebra $B$ is split. If $n$ is even, this is automatically the case, since in the unitary case we are interested in motives modulo~$2$.

For a split algebra $B$ the $J$-invariant was computed by Fino in \cite{Fi19} and we reproduce his formulae in Table~\ref{tab5} below. This finishes the case $^2\A_{2n}$ completely, but the case $^2\A_{2n+1}$ with a non-split algebra $B$ requires an additional consideration.

Let $X$ be the maximal unitary Grassmannian (i.e., the variety of parabolic subgroups of type $\{1,2,\ldots,2n+1\}\setminus\{n+1\}$). Its normed Chow ring was computed by Fino in \cite[Proposition~4.15]{Fi19}. Namely, one has $\CH^*_K(\overline X)=\ff_2[e_1,e_3,\ldots,e_{2n+1}]/(e_1^2,e_3^2,\ldots,e_{2n+1}^2)$. Moreover, it follows from \cite[Theorem~8.1]{Ka12c} that in the generic case the normed Chow motive of $X$ is indecomposable.

To compute the respective ring $C^*$ we consider a generic generically quasi-split variety homogeneous under the group ${\Aut_K(B,\tau)\times R_{K/F}(\PGL_1(B))}$,
which is the product of the maximal unitary Grassmannian and the Weil restriction $R_{K/F}(\SB(B))$ (in particular, the central simple algebra $B$ is generic).

The respective quasi-split group $G^{qs}$ from sequence~\eqref{normedc2} is a central product of a group of type ${^2\A_{2n+1}}$ and a group of type $^2(2\,{{}^1\A_{2n+1}})$.
Since the normed Chow motive of a generic maximal unitary Grassmannian is indecomposable, by Proposition~\ref{polynom} its normed Chow group over a splitting field does not contain rational cycles of positive codimension.
Note also that $\CH^*_K(R_{K/F}(\PP^{2n+1}))=\ff_2[R(h)]/(R(h)^{2n+2})$.
The proof of Proposition~\ref{prounit} below shows
that the subring of rational cycles on $X\times R_{K/F}(\SB(B))$ in the generic case is generated by
the cycle $1\times R(h)^j$ with $j=2^k$ such that $2^k\parallel\frac{2n+2}{2}$ and $\codim R(h)=2$.

Therefore, the ring $C^*$ equals $\ff_2[e_1,e_3,\ldots,e_{2n+1}]/(e_1^2,e_3^2,\ldots,e_{2n+1}^2)\otimes \ff_2[e]/(e^{2^k})$ with $\codim e=2$ and $k$ such that $2^k\parallel\frac{2n+2}{2}$. 

By $v_p$ we denote the $p$-adic valuation, and for a generically quasi-split twisted flag variety $X$ we denote by $\overline{\CH}_K^*(X)$ the subring of rational cycles on $X$ over $F(X)$.

\begin{prop}\label{prounit}
The (normed) $J$-invariant of a central simple algebra $(B,\tau)$ with a unitary involution and such that $\deg B$ is even, is the concatenation of the $J$-invariant of the Hermitian form associated with $\tau_{F(R_{K/F}(\SB(B)))}$ (computed by Fino) and $$j_1=\begin{cases}
v_2(\ind B), & \text{if }\ind B \text{ divides }\frac{\deg B}{2};\\
v_2(\ind B)-1, & \text{otherwise}.
\end{cases}$$
For respective parameters $d_1$ and $k_1$ we have $d_1=2$ and $2^{k_1}\parallel \frac{\deg B}{2}$.
\end{prop}
\begin{proof}
Without loss of generality we can pass to an odd degree field extension and assume that the index of $B$ is a power of $2$. We denote by $A$ the underlying division algebra.

Let $\mathfrak{X}$ be a generically quasi-split variety under the action of the group $\Aut_K(B,\tau)$. Let $U$ be the normed upper motive of $\mathfrak{X}$. Consider the variety ${\mathcal{Y}:=\mathfrak{X} \times R_{K/F}(\SB(A))}$. The variety $\mathcal{Y}$ is generically quasi-split under the action of the group  
$${\Aut_K(B,\tau)\times R_{K/F}(\PGL_1(A))}.$$ Since both varieties $\mathfrak{X}_{F(\mathcal{Y})}$ and $\mathcal{Y}_{F(\mathfrak{X})}$ have rational points, the normed upper motives of $\mathfrak{X}$ and $\mathcal{Y}$ are isomorphic: $\U_K(\mathcal{Y}) \simeq \U_K(\mathfrak{X})=:U$. Denote by $m$ and $F_A$ respectively the degree of $A$ and the function field of the variety $R_{K/F}(\SB(A))$.

Applying Proposition \ref{polynom} to the generically quasi-split varieties $\mathcal{Y}$ and  $\mathfrak{X}_{F_A}$ we obtain
\begin{equation}
\label{pol1}
P(\CH^*_K(\mathcal{Y}),t) = P(\CH^*_K(U),t) \cdot P(\overline{\CH}_K^*(\mathcal{Y}), t)
\end{equation}
\noindent and
\begin{equation}
\label{pol2}
P(\CH^*_K(\mathfrak{X}_{F_A}), t) = P(\CH^*_K(U^{F_A}),t) \cdot P(\overline{\CH}_K^*(\mathfrak{X}_{F_A}), t) \, ,
\end{equation}
where $U^{F_A}$ denotes the normed upper motive of $\mathfrak{X}_{F_A}$.

Note that $P(\CH^*_K(\mathcal{Y}),t) = P(\CH^*_K(\mathfrak{X}),t) \cdot P(\CH^*_K(R_{K/F}(\SB(A))),t) $ and, since over a splitting field of $A$ the variety $\SB(A)$ is isomorphic to $\PP^{m-1}$ and the motive of $\PP^{m-1}$ is a direct sum of Tate motives with the Poincar\'e polynomial
$\frac{t^{m}-1}{t-1}$, we have
by formulae~\eqref{weil1} and \eqref{weil3} that $P(\CH^*_K(R_{K/F}(\SB(A))),t) = \frac{t^{2m}-1}{t^2-1}$.

Dividing equation~\eqref{pol1} by \eqref{pol2} we obtain

\begin{equation}
\label{pol-eq}
\frac{t^{2m}-1}{t^2-1} = \frac{P(\CH^*_K(U),t)}{P(\CH^*_K(U^{F_A}),t)} \cdot \frac{P(\overline{\CH}_K^*(\mathcal{Y}), t)}{P(\overline{\CH}_K^*(\mathfrak{X}_{F_A}), t)} \, . 
\end{equation}

Our next goal is to describe the quotient $P(\overline{\CH}_K^*(\mathcal{Y}), t)/P(\overline{\CH}_K^*(\mathfrak{X}_{F_A}), t)$. We follow similar arguments as in \cite[\S3]{Zh}. We refer to Section~\ref{sec3} for the notation $R(h)$ below.

\begin{lem} 
\label{lemma}
Let $y= a_k \times R(h)^k + \Sum_{i>k} a_i\times R(h)^i$, $a_i \in \CH_K^*(\overline{\mathfrak{X}})$, be an element in $\CH_K^*(\overline{\mathcal{Y}})$, where $h \in  \CH_K^*(\overline{\SB(A)})$ denotes the class of a hyperplane section on $\overline{\SB(A)} \simeq \mathbb{P}^{m-1} $. If $y$ is rational over $F$, then 
\begin{enumerate}
\item $1 \times R(h)^k$ is rational over $F$,
\item $a_k$ is rational over $F_A$.
\end{enumerate}
\end{lem}
\begin{proof} Since both varieties $\mathfrak{X}$ and $R_{K/F}(\SB(A))$ are quasi-split over the function field of $\mathfrak{X}$, we can follow the proof of  \cite[Lemma~3.2]{Zh} by replacing the usual Chow rings $\CH$ by the normed Chow rings $\CH_K$.
\end{proof}

Using Lemma~\ref{lemma}, similarly as in \cite[Proposition~3.4]{Zh} we can show that 
the polynomial $P(\overline{\CH}_K^*(\mathfrak{X}_{F_A}), t)$ divides $P(\overline{\CH}_K^*(\mathcal{Y}), t)$
and the quotient is given by $F(t):=\Sum_{i \in J} t^i $, where

\begin{equation}
\label{1xh}
J=\{0\leq j < m \mid 1\times R(h)^j \text{ is }F\text{-rational in }\CH_K^*(\overline{\mathcal{Y}})    \} \, .
\end{equation}

Now we observe two properties of the set $J$. The first property
\begin{equation}
\label{property1}
x\in J,\, y\in J \text{ and } x+y<m \quad \Longrightarrow\quad  x+y\in J
\end{equation}

\noindent follows from the definition of $J$, since the product of two $F$-rational cycles is $F$-rational.
We also claim that
\begin{equation}
\label{property2}
\text{the set } J \text{ is symmetric with respect to its midpoint.}    
\end{equation}
\noindent To show this property it is equivalent to check that the coefficients of $F(t)$ are symmetric with respect to the midpoint, that is $F(t)= t^{\deg F}F(t^{-1})$. Note that all Poincar\'e polynomials from equality~\eqref{pol-eq}
have symmetric coefficients with respect to the midpoint. Indeed, it is clearly true for $\frac{t^{2m}-1}{t^2-1}$. Since $P(\CH^*_K(M^*),t) =P(\CH^*_K(M),t^{-1})$ for a split motive $M$ and its dual motive $M^*$, it follows from \cite[Proposition~5.2]{Ka10b} that both polynomials $P(\overline{\CH}_K^*(\mathfrak{X}_{F_A}), t)$ and $P(\overline{\CH}_K^*(\mathcal{Y}), t)$ have symmetric coefficients with respect to the midpoint. Therefore, using equality~\eqref{pol-eq}, we have $F(t)= t^{\deg F}F(t^{-1})$ and, thus, the set $J$ is symmetric with respect to its midpoint.

It follows from properties~\eqref{property1}, \eqref{property2} and equality~\eqref{pol-eq} that
$$F(t)= 1+t^{2j}+t^{4j}+ \ldots +t^{2(k-1)j} = \frac{t^{2m}-1}{t^{2j}-1}  ,$$
for some $j$ dividing $m$ and $k=m/j$. Since $m$ is a power of two, we write $j=2^{j_1}$.

Replacing the last quotient in equation~\eqref{pol-eq} by $F(t)$ and using the above formula we obtain
$P(\CH^*_K(U),t) = \frac{t^{2j}-1}{t^2-1} \cdot  P(\CH^*_K(U^{F_A}),t) $.
The Poincar\'e polynomial of the (normed) upper motive $U^{F_A}$ is given by the $J$-invariant  of the Hermitian form associated with $(B_{F_A}, {\tau}_{F_A})$. Concatenating it with the additional component $j_1$ we obtain the $J$-invariant of $(B, \tau)$.

We show next that the parameter $k_1$ of the $J$-invariant satisfies $k_1<v_2(\deg B)$. We assume that we are in type $^2\A_{2n+1}$ and, thus, $\deg B=2n+2$. Consider the generically quasi-split variety $Y$ of parabolic subgroups of the group $\Aut_K(B,\tau)$ of type $$\{1,2,\ldots,2n+1\}\setminus\{1,n+1,2n+1\}.$$ This variety is generically quasi-split by \cite[p.~55, type $^2\A_n$]{Ti66}. Then the normed Poincar\'e polynomial of $Y$ must be divisible by the Poincar\'e polynomial of the normed upper motive $U$ of the variety of Borel subgroups. Thus, the polynomial $$P(\CH^*_K(Y),t)=\frac{P(\CH^*_K({}^2\A_{2n+1}/B),t)}{P(\CH^*_K(R_{K/F}(\A_{n-1}/B')),t)}=\frac{P(\CH^*_K({}^2\A_{2n+1}/B),t)}{P(\CH^*(\A_{n-1}/B')),t^2)}$$
is divisible by the polynomial $P(\CH^*_K(U),t)=(1+t)(1+t^3)\cdots(1+t^{2n+1})\cdot\frac{t^{2\cdot 2^{k_1}}-1}{t^2-1}$.

Computing the quotient $\frac{P(\CH^*_K(Y),t)}{P(\CH^*_K(U),t)}$ we obtain $\frac{t^{2(n+1)}-1}{t^{2\cdot 2^{k_1}}-1}$ which is a polynomial iff $2\cdot 2^{k_1}$ divides $2(n+1)$, i.e. iff $k_1\le v_2(n+1)=v_2(\deg B)-1$.

Denote $v_2(\ind A)$ by $i_A$. We claim that $j_1=\min \{k_1, i_A\}$ and $k_1=v_2(\deg B)-1$. We know that
$\min \{k_1, i_A\}$ is an upper bound for $j_1$, hence, it is sufficient to show that 
$j_1 \geq \min \{k_1, i_A\}$.  Recall that $j=2^{j_1}= \min \, J\backslash \{0\}$. It follows from Definition~\ref{1xh} of the set $J$ that the cycle $1\times R(h)^{j} \in \CH_K^*(\overline{\mathfrak{X} \times R_{K/F}(\SB(A))})$ is $F$-rational. Then the cycle 
$1\times 1\times R(h)^{j} \in \CH_K^*(\overline{\mathfrak{X} \times \mathcal{Z} \times R_{K/F}(\SB(A)))}$ is also $F$-rational, where $\mathcal{Z}$ is the variety of maximal parabolic subgroups of $\Aut_K(B,\tau)$ of type $\{1,2,\ldots, 2n+1\}\setminus\{n+1\}$, i.e., the variety of totally isotropic ideals in $B$ of reduced dimension $n+1$.
    
Note that the variety $\mathfrak{X}$ possesses  a rational point over the function field of the variety $\mathcal{Z} \times R_{K/F}(\SB(A))$. Therefore, by \cite[Lemma~2.3]{Zh} the cycle $${1\times R(h)^{j} \in \CH_K^*(\overline{ \mathcal{Z} \times R_{K/F}(\SB(A))})}$$ is $F$-rational. It follows that ${R(h)^{j} \in \CH_K^*(\overline{ R_{K/F}(\SB(A))})}$ is rational over $F(\mathcal{Z})$.
Therefore, passing to $K$ we find a rational cycle over $K(\mathcal{Z})$ of the form $${h^j\times h^j+\alpha\in\Ch^*(\overline{\SB(A)\times\SB(A)})}$$ in the ordinary Chow group, where $\alpha$ is some cycle which goes to zero in the respective normed Chow group.

We want to show that $j  \geq \ind A_{K(\mathcal{Z})}$. In general, if $D$ is a central simple algebra,
the first projection $\SB(D)\times\SB(D)\to\SB(D)$ is a projective bundle over $\SB(D)$. This allows to give an explicit combinatorial description of all rational cycles in $\Ch^*(\overline{\SB(D)\times\SB(D)})$ in terms of the index $\ind D$.

It follows in our situation that if $j<\ind A_{K(\mathcal{Z})}$,
the cycle $h^j\times h^j+\alpha$ cannot be rational independently of $\alpha$ as above.
Thus, $j  \geq \ind A_{K(\mathcal{Z})}$.

Note that $\mathcal{Z}_K \simeq \SB_{n+1}(B)$ and, thus, $K(\mathcal{Z}) \simeq K(\SB_{n+1}(B))$.
Then applying the Index Reduction Formula \cite[page~592]{MPW}, we obtain that $\ind A_{K(\mathcal{Z})} = \gcd\{n+1, \ind A \}$. It follows that $j_1 \geq \min \{k_1, i_A\}$ and $k_1=v_2(n+1)$.
\end{proof}

\begin{cor}\label{cor67}
In the notation of Proposition~\ref{prounit} ($\deg B$ is even) we have: $j_2=0$ or~$1$, and
$j_2=0$ iff the discriminant algebra $D(B,\tau)$ (see \cite{Ta76}, \cite[Definition~(10.28)]{KMRT}) is split.
\end{cor}
\begin{proof}
Let $m=\deg B$. For $j_2$ let $Y$ denote the variety of Borel subgroups of the group $\PGU(B,\tau)$ and consider the exact sequence (see \cite[\S2 and Corollary~2.8]{MT95})
$$0\to\Pic(Y)\xrightarrow{\res}\Pic(\overline Y)^\Gamma\xrightarrow{\alpha_Y}\Br(F),$$
where $\Gamma$ is the absolute Galois group and $\alpha_Y$ is the Tits map. Then the middle fundamental weight $\omega_{m/2}\in\Pic(\overline Y)^\Gamma$ maps to the discriminant algebra $D(B,\tau)$.

Therefore, $\omega_{m/2}$ is rational if and only if $D(B,\tau)$ is split. Since generic points of direct summands of the normed motive of $Y$ are precisely rational cycles and since $j_2$ has degree one (i.e., $d_2=1$), it follows that $j_2=0$ if and only if the group $\Pic(\overline Y)^\Gamma$ is rational, i.e. by the above considerations if and only if $D(B,\tau)$ is split.
\end{proof}

\begin{cor}
In the notation of Proposition~\ref{prounit} ($\deg B=2n+2$ is even) let $$(j_1,j_2,j_3,\ldots,j_{n+2})$$ be the $J$-invariant of $(B,\tau)$. Then the $J$-invariant of $(B,\tau)_{F(\SB(D(B,\tau)))}$ equals $$(j_1,0,j_3,\ldots,j_{n+2}),$$ i.e., apart from the value of $j_2$ the $J$-invariant does not change. 
\end{cor}
\begin{proof}
The proof is the same as of \cite[Theorem~4.1(1)]{Zh}.
\end{proof}

{\bf Orthogonal groups.}

We consider now the case $^2\D_n$. This case corresponds to a central simple algebra $A$ of degree $2n$ over $F$ with an orthogonal involution $\sigma$ with a non-trivial discriminant (if $\Char F=2$, then one should be more careful here. We omit the details and refer to \cite{KMRT}). As in the unitary case we assume first that the algebra $A$ is split. This case corresponds to a $2n$-dimensional quadratic form of a non-trivial discriminant. 

By \cite[Theorem~5.1]{Ka12a} the ordinary Chow motive modulo $2$ of the generic orthogonal Grassmannian $\OGr$ of isotropic $(n-1)$-dimensional subspaces (it corresponds to the parabolic subgroup of type $\{1,2,\ldots,n-2\}$) is indecomposable and the same method of the proof shows that its normed motive is indecomposable as well.  Since for a split algebra $A$ the upper motive of this Grassmannian is isomorphic to the upper motive of the variety of Borel subgroups, its Poincar\'e polynomial
equals the Poincar\'e polynomial of $^2\D_n/P_{\{1,2,\ldots,n-2\}}$.
The formula for the latter Poincar\'e polynomial follows directly from Proposition~\ref{solo}, namely, it equals
\begin{equation}\label{eqq}
\begin{aligned}
&P(\CH_K^*({}^2\D_n/B),t)/P(\CH_K^*({}^1\A_{n-2}/B),t)=\\
&\Big(\tfrac{t^2-1}{t-1}\cdot\tfrac{t^4-1}{t-1}\cdots\tfrac{t^{2n-2}-1}{t-1}\cdot\tfrac{t^n+1}{t+1}\Big)\cdot\Big(\tfrac{t-1}{t^2-1}\cdot\tfrac{t-1}{t^3-1}\cdots\tfrac{t-1}{t^{n-1}-1}\Big)=\prod_{k=2}^{n} (x^k+1).
\end{aligned}
\end{equation}

Since the motive of the generic orthogonal Grassmannian of isotropic $(n-1)$-dimensional subspaces is indecomposable, it follows that the ring $C^*$ from sequence~\eqref{normedc2} equals in the $^2\D_n$ case the normed Chow ring $\CH^*_K(\OGr(n-1,q))$ for a quasi-split quadratic form $q$ of dimension $2n$.

In \cite[Proposition~2.11]{Vi07} Vishik gives an explicit description of the Chow rings of all orthogonal Grassmannians. It follows immediately from his description that

\begin{equation}\label{ogr}
\CH^*_K(\OGr(n-1,q))\simeq\ff_2[e_2,e_{3},e_5,\ldots,e_{2[\frac{n+1}{2}]-1}]/(e_2^{2^{k_2}},e_3^{2^{k_3}},e_5^{2^{k_4}},\ldots,e_{2r-3}^{2^{k_r}})
\end{equation}
with $\codim e_l=l$, $r=[\frac{n+1}{2}]+1$, $k_2=[\log_2 n]$ and $k_i=[\log_2\frac{2n}{2i-3}]$, $i\ge 3$,
and this gives the respective parameters of the $J$-invariant for type $^2\D_n$ when the algebra $A$ is split.

Therefore, it remains to consider the case when the algebra $A$ is arbitrary.

Similarly to the unitary case we consider a generic generically quasi-split variety homogeneous under the group ${\PGO^+(A,\sigma)\times \PGL_1(A)}$ (in particular, the central simple algebra $A$ is generic),
which is the product of the orthogonal Grassmannian of parabolic subgroups of type $\{1,2,\ldots,n-2\}$ considered above and $\SB(A)$.

The respective quasi-split group $G^{qs}$ from sequence~\eqref{normedc2} is a central product of a group of type $^2\D_n$ and a group of type $^1\A_{2n-1}$. Since the normed Chow motive of a generic orthogonal Grassmannian $X$ of parabolic subgroups of type $\{1,2,\ldots,n-2\}$ is indecomposable,
its normed Chow group over a splitting field does not contain rational cycles of positive codimension.
Note also that $\CH^*_K(\PP^{2n-1})=\ff_2[h]/(h^{2n})$.
The proof of Proposition~\ref{prop69} below shows
that the subring of rational cycles on $X\times \SB(A)$ in the generic case is generated by
the cycle $1\times h$, if $n$ is even and by $1\times h^2$, if $n$ is odd,
and $\codim h=1$.

Therefore, the ring $C^*$ equals
\begin{equation}\label{eq714}
\ff_2[e_2,e_{3},e_5,\ldots,e_{2[\frac{n+1}{2}]-1}]/(e_2^{2^{k_2}},e_3^{2^{k_3}},e_5^{2^{k_4}},\ldots,e_{2r-3}^{2^{k_r}})\otimes \ff_2[e]/(e^{2^k})
\end{equation}
with $\codim e=1$, $k=\frac{1+(-1)^{n+1}}{2}$, and $r$ and $k_i$ as in formula~\eqref{ogr}.

\begin{prop}\label{prop69}
The (normed) $J$-invariant of a central simple algebra with an orthogonal involution $(A,\sigma)$ of a non-trivial discriminant is the concatenation of the $J$-invariant of the quadratic form associated with $\sigma_{F(\SB(A))}$ (computed above) and $${j_1=\min\{k_1,v_2(\ind A)\}}.$$

For respective parameters $d_1$ and $k_1$ we have $d_1=1$ and $k_1=\frac{1+(-1)^{n+1}}{2}$, where $2n$ is the degree of $A$.
\end{prop}
\begin{proof}
Let $U$ denote the normed upper motive of a generically quasi-split variety for the group $\PGO^+(A,\sigma)$ and $U^{F_A}$ the normed upper motive of the same variety over the function field $F_A$ of the Severi--Brauer variety $\SB(A)$.
Proceeding as in \cite[\S3]{Zh} we arrive at the following identity for the Poincar\'e polynomials $$P(\CH^*_K(U),t) = \frac{t^{j}-1}{t-1} \cdot  P(\CH^*_K(U^{F_A}),t), $$
where $j=2^{j_1}$ for $j_1$ we are looking for. Moreover, the maximal possible value of $j_1$ under all groups of type $^2\D_n$ is precisely the parameter $k=k_1$ from formula~\eqref{eq714}.

We show now that $2^{k_1}\mid(\frac{\deg A}{2}-1)=n-1$. Consider the generically quasi-split variety $Y$ of parabolic subgroups of type $\{1,2,\ldots,n\}\setminus\{1,n-1,n\}$. This variety is generically quasi-split by \cite[p.~57, type $^2\D_n$]{Ti66}. Then the normed Poincar\'e polynomial of $Y$ must be divisible by the Poincar\'e polynomial of the normed upper motive $U$ of the variety of Borel subgroups. Thus, the polynomial $$P(\CH^*_K(Y),t)=\frac{P(\CH^*_K({}^2\D_{n}/B),t)}{P(\CH^*_K(\A_{n-3}/B'),t)}$$
is divisible by the polynomial $P(\CH^*_K(U),t)=(1+t^2)(1+t^3)\cdots(1+t^n)\cdot\frac{t^{2^{k_1}}-1}{t-1}$.

Computing the quotient $\frac{P(\CH^*_K(Y),t)}{P(\CH^*_K(U),t)}$ we obtain $\frac{t^{n-1}-1}{t^{2^{k_1}}-1}$ which is a polynomial iff $2^{k_1}$ divides $n-1$.

In particular, if $n$ is even, then $k_1=0$. On the other hand, if $n$ is odd, then the index of the algebra $A$ is at most $2$. Therefore, if $n$ is odd, then $k_1\le 1$ and as in the proof of Corollary~\ref{cor67} using \cite[Corollary~2.11]{MT95} one can see that $k_1\ne 0$, since the normed Picard group of the variety of Borel subgroups is not rational. Therefore, $k_1=1$ if $n$ is odd.
\end{proof}

\begin{rem}
Combining formulae~\eqref{eqq} and \eqref{ogr} we have the following combinatorial identity:
$$\sum_{i=1}^{[\frac{n+1}{2}]-1}[\log_2\tfrac{2n}{2i+1}]=n-1-[\log_2 n].$$
\end{rem}

\begin{rem}
If our group has type $^2\D_n$ with an even $n$, then by the classification of Tits indices \cite[p.~57, type $^2\D_n$]{Ti66} the variety of its parabolic subgroups of type $\{1,2,\ldots,n-2\}$ is generically quasi-split.

We remark also aside that a similar situation occurs in the unitary case. If $n$ is even, then the variety of parabolic subgroups of type $\{1,2,\ldots, n\}\setminus\{\frac{n}{2},\frac{n}{2}+1\}$ is generically quasi-split for groups of type $^2\A_n$. This is the maximal unitary Grassmannian.
\end{rem}

\begin{rem}
In \cite{Vi05} Vishik defined the $J$-invariant of quadratic forms.  Note that our version of the $J$-invariant is not equivalent to the $J$-invariant of Vishik in case when the quadratic form has a non-trivial discriminant.

Consider an anisotropic excellent even-dimensional quadratic form $q$ with a non-trivial discriminant over a field of characteristic different from $2$. This form is given by a strictly decreasing sequence of embedded Pfister forms $\pi_0 \supset \pi_1 \supset \ldots \supset \pi_s$, where $s$ is a positive integer and
$\dim(\pi_{s-1}) > 2 \dim(\pi_s)$. Since the discriminant of $q$ is non-trivial, the last Pfister form $\pi_s$ is a $1$-fold Pfister form given by the discriminant of $q$.

The normed $J$-invariant of the respective group of 
type $^2\D$ equals $(j_1,j_2,\ldots,j_r)$, where $j_i$ of degree $d_i$ equals $1$, if $d_i=\dim\pi_{s-1}/2-1$ and equals $0$ otherwise.

Note that the $J$-invariant defined in \cite{Vi05} is not equivalent to our $J$-invariant in this case, since the quadratic form $q$ is divisible by the binary form $\pi_s$ given by its discriminant.
\end{rem}

The following result is a counterpart of Corollary~\ref{cor67} for type $^2\D_n$. Note that opposite to the unitary case the degree of $j_2$ is $d_2=2$.

\begin{cor}\label{cor718}
In the above notation we have:
$j_2=0$ iff the (even) Clifford algebra $C(A,\sigma)$ is split.
\end{cor}
\begin{proof}
First we pass to the function field $E:=F(\SB(A))$. By Proposition~\ref{prop69} the value of $j_2$ does not change, but the value of $j_1$ becomes zero.

We claim that the Clifford algebra $C(A,\sigma)$ is split over $K$ iff the Clifford algebra $C(A_{E},\sigma_{E})$ is split over $K(\SB(A_K))$. Indeed, by the index reduction formula \cite{SV92} we have $$\ind C(A_E,\sigma_E)=\ind C(A,\sigma)_{K(\SB(A_K))}=\gcd_{1\le i\le \deg A} \ind(C(A,\sigma)\otimes_K A_K^{\otimes i}).$$
We have that $\exp A\mid 2$ and by \cite[Proof of Theorem~9.12]{KMRT} $C(A,\sigma)\otimes_K A_K$ is Brauer-equivalent to the conjugate algebra $^\iota C(A,\sigma)$ defined in \cite[Section~3.B]{KMRT}. Moreover, the algebras $C(A,\sigma)$ and $^\iota C(A,\sigma)$ are split or not split simultaneously. Therefore,
$\ind C(A_E,\sigma_E)=1$ iff $\ind C(A,\sigma)=1$.

Thus, we can assume without loss of generality that the algebra $A$ is split over $F$.
This case corresponds to a quasi-split group $G^{qs} = \PGO^+(q)$ for some $2n$-dimensional quadratic form  $q$ with a non-trivial discriminant. 

Assume now that the Clifford algebra is split as well. Then the orthogonal Grassmannian $X=\OGr(n-2, q)$ is a generically quasi-split $G^{qs}$-variety. Computing  the normed Poincar\'e polynomial by Proposition \ref{solo} we obtain 
$$P(\CH^*_K(X),t)=\frac{P(\CH^*_K({}^2\D_{n}/B),t)}{P(\CH^*_K(\A_{n-3}/B'),t) P(\CH^*_K(^2(2\,{\A_{1}})/B''),t)} = 1+t+t^2 + \ldots \, ,$$
\noindent where $`` \ldots "$ stands for a sum of monomials of degree $\geq 3$.

Since $A$ is split, note that the generator $w$ of $\CH^1_K(\overline{X}) \simeq \ff_2$ is $F$-rational, and moreover $w^2 \neq 0$. Hence, we have $P(\overline{\CH}_K^*(X), t) = 1+t+t^2 + \ldots \, $, where  $ `` \ldots "$ stands for a sum of monomials of degree $\geq 3$.

Assume that $j_2 \neq 0$. Then by formula~\eqref{J-formula} below (note that the proof of formula~\eqref{J-formula} is independent of the present corollary) the normed Poincar\'e polynomial of the upper motive $U$ of $X$ over a splitting field has the factor
$$ \frac{(t^{2})^{2^{j_2}}-1}{t^{2}-1} = 1+t^2 + \ldots \, , $$
\noindent where $ `` \ldots "$ stands for a sum of monomials of degree $\geq 3$.
Therefore,  we have $$P(\CH^*_K(U),t) \cdot P(\overline{\CH}_K^*(X), t) = 1 +t + 2t^2 + \ldots \, ,$$ and this contradicts Proposition~\ref{polynom}. Thus, $j_2 =0$.

Conversely, assume that $A$ is split and $j_2=0$.
We consider the splitting tower of the respective quadratic form $q$. The last anisotropic quadratic form in the splitting tower is a binary form given by the discriminant of $q$. We pass to the penultimate step in the splitting tower.

Assume that the Clifford algebra of $q$ is not split.
Note that by the index reduction formula the Clifford algebra remains non-split, if we pass to the function field of any quadratic form of dimension bigger than $4$. Therefore, we can assume that the penultimate quadratic form in the splitting tower of $q$ is $4$-dimensional. But then the respective generically quasi-split variety is just a $2$-dimensional quadric. But if $j_2=0$, then $\Ch^2$ of this quadric is rational, and this quadric cannot be anisotropic by the Springer theorem.
\end{proof}

\begin{cor}
In the above notation let $(j_1,j_2,j_3,\ldots,j_{[\frac{n+1}{2}]+1})$ be the $J$-invariant of $(A,\sigma)$. Then the $J$-invariant of $(A,\sigma)_{F(R_{K/F}(\SB(C(A,\sigma))))}$ equals $(j_1,0,j_3,\ldots,j_{[\frac{n+1}{2}]+1})$, i.e., apart from the value $j_2$ the $J$-invariant does not change.
\end{cor}
\begin{proof}
The proof is the same as of Proposition~\ref{prounit}.
\end{proof}

Summarizing above considerations we see that the ring $C^*$ from sequence~\eqref{normedc2} has presentation~\eqref{pres} for all absolutely simple quasi-split algebraic groups. By the K\"unneth formula the same holds for products of absolutely simple quasi-split groups. Moreover, the same holds for the Weil restrictions of absolutely simple quasi-split groups. Indeed, let $G$ be quasi-split absolutely simple algebraic group over a field extension $L$ of the base field $F$ and let $B$ be a Borel subgroup of $G$ defined over $L$.

Recall that $[K:F]=p^n$ (where $K$ is the field from the definition of normed Chow rings such that the group $G_K$ is of inner type) for some prime number $p$ and for some $n\ge 1$.
If the group $G$ is split over $L$, then by the properties of the Weil restrictions (see Section~\ref{sec3}) the functor $R_{L/F}$ induces an isomorphism $\Ch^*(G/B)\xrightarrow{\simeq} \CH^*_K(R_{L/F}(G/B))$. Finally, if $G$ is a not split over $L$, then the functor $R_{L/F}$ induces an isomorphism  $\CH^*_K(G/B)\xrightarrow{\simeq} \CH^*_K(R_{L/F}(G/B))$.

Thus, the ring $\CH^*_K(R_{L/F}(G/B))$ has presentation~\eqref{pres}.
Therefore, the ring $C^*$ has presentation~\eqref{pres} for all adjoint semisimple groups. Since every semisimple group is a twist of a quasi-split group by means of a cocycle with values in an adjoint semisimple quasi-split group, we can now define the (normed) $J$-invariant for arbitrary semisimple groups (more precisely, for torsors under an adjoint semisimple quasi-split group) exactly as in Definition~\ref{def71} with Chow rings replaced by the normed Chow rings and with $\Ch^*(G)$ replaced by $C^*$.

Now we are ready to fill the following table which contains the parameters of the $J$-invariant for absolutely simple groups of outer type.

\begin{center}
\begin{longtable}{l|l|l|l|l|l}
\caption{Parameters of the $J$-invariant for normed motives (outer type)}\label{tab5}\\
Adjoint group $G$ & $p$ & $r$ & $d_i$, $i=1\ldots r$ & $k_i$, $i=1\ldots r$ & Restrictions on $j_i$\\
\hline
${}^2\A_{2n}$       & $2$ & $n$           &  $2i+1$  & $1$ & if $d_i+l=d_m$ and\\
&&&&&$2\nmid\binom{d_i}{l}$, then $j_m \le j_i$\\
${}^2\A_{2n+1}$            & $2$       & $n+2$      & $2,\,i=1$ & $2^{k_1}\parallel (n+1)$         & the same restrictions \\
     &  &                 & $2i-3,\,i\ge 2$  &$1$   & assuming $i,m\ge 2$\\
${}^2\D_{n}$       & $2$ & $[\frac{n+1}{2}]+1$           &  $1,\,i=1$  & $\tfrac{1+(-1)^{n+1}}{2}$ &  if $d_i + l = 2^sd_m$ and \\
&  &                 & $2,\,i=2$ & $[\log_2 n]$ & $2\nmid\binom{d_i}{l}$, then $j_m \le j_i+s$\\
&  &                 & $2i-3,\,i\ge 3$  &$[\log_2\frac{2n}{2i-3}]$   & assuming $i,m\ge 2$\\
${}^3\D_4$
& $3$       & $1$            & $4$  & $1$                          & \\
${}^2\E_6$                   & $2$       & $3$                  & $3,\,5,\,9$ & $1,\,1,\,1$    & $j_1\ge j_2\ge j_3$\\
$R_{L/F}(G)$ & $p$ & $r(G)$ & $[L:F]\cdot d_i(G)$ & $k_i(G)$ & as for $G$\\
\end{longtable}
\end{center}

The restrictions on $j_i$ in the last column of the table follow from \cite[Theorem~7.2]{Fi19}, \cite[Proposition~2.9]{Vi07} and for the type ${}^2\E_6$ from Table~\ref{jinv} (type $\E_7$).

Moreover,
the following proposition can be proved exactly in the same way as \cite[Theorem~5.13]{PSZ08}.

\begin{thm}\label{prop612}
Let $G_0$ be a quasi-split semisimple algebraic group over a field $F$, let $K$ be a minimal field extension such that the group $(G_0)_K$ is of inner type, let $B$ be a Borel subgroup of $G_0$, let $E$ be a $G_0$-torsor over $F$ and let $p$ be a torsion prime of $G_0$. Assume that $[K:F]=p^n$ for some $n\ge 1$.

Then the normed Chow motive of $E/B$ decomposes into a direct sum of Tate twists of an indecomposable motive $R_p(E)$, and the Poincar\'e polynomial of $R_p(E)$ over a splitting field of $E$ equals
\begin{equation}
\label{J-formula}
\prod_{i=1}^r\frac{t^{d_ip^{j_i}}-1}{t^{d_i}-1},
\end{equation}
where the parameters $r$ and $d_i$ are taken from Table~\ref{tab5} and $(j_1,\ldots,j_r)$ is the $J$-invariant of $E$.
\end{thm}

\begin{rem}
Exactly the same statement holds, if we replace $B$ by a parabolic subgroup $P$ such that the variety $E/P$ is generically quasi-split.
\end{rem}


\begin{thebibliography}{MPW96}

\bibitem[Bor]{Bor} {\sc A.~Borel}, Linear algebraic groups. Second Edition, Springer-Verlag New-York Inc., 1991.

\bibitem[Br99]{Br99} {\sc P.~Brosnan}, Steenrod operations in Chow theory, {\it Trans. AMS} {\bf 355} (1999), 1869--1903.

\bibitem[Car]{Car} {\sc R.~Carter}, Simple groups of Lie type, John Wiley \& Sons, 1972.

\bibitem[Ch03]{Ch03} {\sc V. Chernousov}, The kernel of the Rost invariant, Serre's Conjecture II and the Hasse principle for quasi-split groups $^{3,6}\D_4$, $\E_6$, $\E_7$, {\it Math. Ann.} {\bf 326} (2003), 297--330.

\bibitem[De74]{De74} {\sc M. Demazure}, {D\'esingularisation des vari\'et\'es de Schubert g\'en\'eralis\'ees}, {\it Ann. Sci. \'Ecole Norm. Sup.} \textbf{7} (1974), 53--88.

\bibitem[EG98]{EG98} {\sc D.~Edidin, W.~Graham}, Equivariant intersection theory (with an appendix by A.~Vistoli: The Chow ring of $\mathcal M_2$), {\it Invent. Math.} {\bf 131} (1998), no.~3, 595--634.


\bibitem[EKM]{EKM} {\sc R.~Elman, N.~Karpenko, A.~Merkurjev}, The algebraic and geometric theory of quadratic forms,
{\em Colloquium Publications}, vol.~56, American Mathematical Society, Providence, RI, 2008.

\bibitem[Fi19]{Fi19} {\sc R.~Fino}, $J$-invariant of hermitian forms over quadratic extensions,
{\it Pacific J. of Math.} {\bf 300} (2019), no.~2, 375--404.

\bibitem[Ful]{Ful} {\sc W.~Fulton}, Intersection theory. Second Edition,
Ergebnisse der Mathematik und ihrer Grenz\-gebiete, 3. Folge. A Series of Modern Surveys in Mathematics {\bf 2},
Springer-Verlag, Berlin, 1998.

\bibitem[Ga01]{Ga01} {\sc S.~Garibaldi}, The Rost invariant has trivial kernel for quasi-split groups of low rank, {\it Comm. Math. Helv.} {\bf 76} (2001), 684--711.

\bibitem[GMS03]{GMS03}
{\sc S.~Garibaldi, A.~Merkurjev, J-P.~Serre},
Cohomological invariants in Galois cohomology,
AMS, Providence, RI, 2003.

\bibitem[GPS16]{GPS16}
{\sc S.~Garibaldi, V.~Petrov, N.~Semenov},
Shells of twisted flag varieties and the Rost invariant,
{\it Duke Math. J.} {\bf 165} (2016), no.~2, 285--339. 

\bibitem[GS10]{GS10} {\sc S.~Garibaldi, N.~Semenov}, Degree $5$ invariant of $\mathrm E_8$, {\it Int. Math. Res. Not.} 2010, no.~19, 3746--3762.

\bibitem[GiZ12]{GiZ12} {\sc S.~Gille, K.~Zainoulline}, {Equivariant pretheories and invariants of torsors}, {\it Transf. Groups} \textbf{17} (2012), 471--498.

\bibitem[Gr58]{Gr58}
{\sc A.~Grothendieck}, La torsion homologique et les sections rationnelles,
Expos\'e 5 in Anneaux de Chow et applications, S\'eminaire Claude Chevalley, 1958.

\bibitem[Kac85]{Kac85} {\sc V.~Kac}, Torsion in cohomology of compact Lie groups and Chow rings of reductive algebraic groups, {\it Invent. Math.} {\bf 80} (1985), 69--79.

\bibitem[Ka00]{Ka00} {\sc N.~Karpenko}, Weil transfer of algebraic cycles, {\it Indag. Math.} {\bf 11} (2000), no.~1, 73--86.

\bibitem[Ka10a]{Ka10} {\sc N.~Karpenko}, Upper motives of outer algebraic groups, In Quadratic forms, linear algebraic groups, and cohomology, {\it Dev. Math.} {\bf 18} (2010) Springer, New York, 249--258.

\bibitem[Ka10b]{Ka10b} {\sc N.~Karpenko}, Canonical dimension, {\it Proceedings of the International Congress of Mathematicians}  (2010), volume II (New Delhi), Hindustan Book Agency, 146--161.

\bibitem[Ka12a]{Ka12a} {\sc N.~Karpenko}, Sufficiently generic orthogonal grassmannians, {\it J. Algebra} {\bf 372} (2012) 365--375.

\bibitem[Ka12b]{Ka12b} {\sc N.~Karpenko}, Incompressibility of quadratic Weil transfer of generalized Severi-Brauer varieties, {\it J. Inst. Math. Jussieu} {\bf 11} (2012), no.~1, 119--131.

\bibitem[Ka12c]{Ka12c} {\sc N.~Karpenko}, Unitary grassmannians, {\it J. Pure Appl. Algebra} {\bf 216} (2012), no.~12, 2586--2600.

\bibitem[Ka13]{Ka13} {\sc N.~Karpenko}, Upper motives of algebraic groups and incompressibility of Severi--Brauer varieties, {\it J. Reine Angew. Math.} {\bf 677} (2013), 179--198.

\bibitem[Ka15]{Ka15} {\sc N.~Karpenko}, Incompressibility of products of Weil transfers of generalized Severi--Brauer varieties, {\it Math. Z.} {\bf 279} (2015), no.~3--4, 767--777.

\bibitem[KM06]{KM06} {\sc N.~Karpenko, A.~Merkurjev},
Canonical $p$-dimension of algebraic groups,
{\it Adv. Math.} {\bf 205} (2006), 410--433. 

\bibitem[KaZ13]{KaZ13} {\sc N.~Karpenko, M.~Zhykhovich}, Isotropy of unitary involutions, {\it Acta Math.} {\bf 211} (2013), no.~2, 227--253.

\bibitem[KMRT]{KMRT}
{\sc M-A.~Knus, A.~Merkurjev, M.~Rost, J-P.~Tignol},
The book of involutions,
{\em Colloquium Publications}, vol.~44, AMS 1998.

\bibitem[LM]{LM} {\sc M.~Levine, F.~Morel}, Algebraic cobordism, {\it Springer Monographs in Mathematics}, Springer-Verlag Berlin Heidelberg (2007).

\bibitem[Ma68]{Ma68} {\sc Y.~Manin}, Correspondences, motives and monoidal transformations,
{\em Math. USSR Sbornik} {\bf 6} (1968), 439--470.

\bibitem[MPW96]{MPW}
{\sc A.~Merkurjev, I. Panin, A. Wadsworth},
\newblock Index reduction formulas for twisted flag varieties. I.
\newblock {\em K-Theory} {\bf 6} (1996), issue~3, 517--596.

\bibitem[MT95]{MT95} {\sc A.~Merkurjev, J-P.~Tignol}, The multipliers of similitudes and the
Brauer group of homogeneous varieties, {\it J. Reine Angew. Math.} {\bf 461}, (1995), 13--47.

\bibitem[NSZ09]{NSZ09}
{\sc S.~Nikolenko, N.~Semenov, K.~Zainoulline},
Motivic decomposition of anisotropic varieties of type $\mathrm{F}_4$ into generalized Rost motives,
{\em J. of K-theory} {\bf 3} (2009), no.~1, 85--102.

\bibitem[Pa94]{Pa94} {\sc I.~Panin}, On the algebraic $K$-theory of twisted flag varieties, {\it K-Theory} {\bf 8} (1994), no.~6, 541--585.

\bibitem[PS10]{PS10}
{\sc V.~Petrov, N.~Semenov}, Generically split projective homogeneous varieties,
{\em Duke Math. J. } {\bf 152} (2010), 155--173.

\bibitem[PS12]{PS12}
{\sc V.~Petrov, N.~Semenov},
Generically split projective homogeneous varieties. II,
{\it J. of K-theory} {\bf 10} (2012), issue~1, 1--8.

\bibitem[PS17]{PS17} {\sc V.~Petrov, N.~Semenov}, Rost motives, affine varieties, and classifying spaces, {\it Journal of London Math. Soc.} {\bf 95} (2017), issue~3, 895--918.

\bibitem[PS21]{PS21} {\sc V.~Petrov, N.~Semenov},
Hopf-theoretic approach to motives of twisted flag varieties, {\it Compos. Math.} {\bf 157} (2021), no.~5, 963--996.

\bibitem[PSZ08]{PSZ08}
{\sc V.~Petrov, N.~Semenov, K.~Zainoulline},
$J$-invariant of linear algebraic groups, {\em Ann. Sci. \'Ecole Norm. Sup.} {\bf 41} (2008), 1023--1053.

\bibitem[Pr20]{Pr20} {\sc E. Primozic}, Motivic Steenrod operations in characteristic $p$, {\it Forum Math., Sigma} {\bf 8} (2020), e52, 1--25.

\bibitem[Ro98]{Ro98} {\sc M.~Rost}, The motive of a Pfister form, Preprint 1998, available from {\tt http://www.math.uni-bielefeld.de/\~{}rost}

\bibitem[SV92]{SV92} {\sc A.~Schofield, M.~Van den Bergh}, The Index of a Brauer class on a Brauer--Severi variety, {\it Trans. AMS} {\bf 333} (1992), no.~2, 729--739.

\bibitem[S16]{S16} {\sc N.~Semenov}, Motivic construction of cohomological invariants, {\it Comment. Math. Helv.} {\bf 91}, issue~1 (2016), 163--202. 

\bibitem[Spr]{Spr} {\sc T.A.~Springer}, Linear algebraic groups. Second Edition, Birkh\"auser Boston, 2009.

\bibitem[Ta76]{Ta76} {\sc T.~Tamagawa},
Representation theory and the notion of the discriminant, {\it Algebraic number theory} (Kyoto Internat. Sympos., Res. Inst. Math. Sci., Univ. Kyoto, Kyoto, 1976), 219--227. Japan Soc. Promotion Sci., Tokyo, 1977.


\bibitem[Ti66]{Ti66} {\sc J.~Tits}, Classification of algebraic semisimple groups,
In {\it Algebraic Groups and Discontinuous Subgroups} (Proc. Symp. Pure Math.),
Amer. Math. Soc., Providence, R.I., 1966, 33--62.

\bibitem[To99]{To99} {\sc B.~Totaro}, The Chow ring of a classifying space, in {\it Algebraic K-Theory}, ed. W.~Raskind and C.~Weibel, Proceedings of Symposia in Pure Mathematics, vol. 67, American Mathematical Society, 1999, 249--281.


\bibitem[Vi05]{Vi05} {\sc A.~Vishik}, {On the Chow groups of quadratic Grassmannians}, {\it Doc. Math.} \textbf{10} (2005), 111--130.

\bibitem[Vi07]{Vi07}
 {\sc A. Vishik}, Fields of $u$-invariant $2^r + 1$, In: Algebra, Arithmetic and Geometry, Manin
Festschrift, Birkh\"auser (2007).

\bibitem[Vi24]{Vi24}
{\sc A.~Vishik}, Isotropic and numerical equivalence for Chow groups and Morava $K$-theories, {\it Inventiones Math.} {\bf 237} (2024), 779--808.

\bibitem[Vo03a]{Vo03a} {\sc V.~Voevodsky},
Reduced power operations in motivic cohomology, {\it Publ. Math.,
Inst. Hautes \'Etud. Sci.} {\bf 98} (2003), 1--57.

\bibitem[Vo03b]{Vo03b} {\sc V.~Voevodsky}, {Motivic cohomology with $\zz/2$-coefficients}, {\it Publ. Math.,
Inst. Hautes \'Etud. Sci.} {\bf 98} (2003), 59--104.

\bibitem[Zh24]{Zh} {\sc M.~Zhykhovich}, The $J$-invariant over splitting fields of Tits algebras, {\it Comp. Math.} {\bf 160} (2024), issue~9, 2100--2114.


\end{thebibliography}
\end{document}